\documentclass[11pt]{article}

\usepackage{amsfonts}
\usepackage{amsmath}
\usepackage{amssymb}
\usepackage{wasysym}
\usepackage{yfonts}  
\usepackage{amsthm}
\usepackage{amscd}   
\usepackage[all]{xypic}
\usepackage{mathrsfs}
\usepackage{indentfirst}
\usepackage{bbold}
\usepackage{amsxtra}
\usepackage{amsbsy}
\usepackage{amstext}
\usepackage{amsopn}
\usepackage{enumerate}
\usepackage{mathtools}
\usepackage{relsize}

\newenvironment{enumerate*}{
\begin{enumerate}[{\rm (i)}]
  \setlength{\itemsep}{3.5pt}
  \setlength{\parskip}{0pt}
}{\end{enumerate}}

\newenvironment{enumerate!}{
\begin{enumerate}[{\rm (I.)}]
  \setlength{\itemsep}{3.5pt}
  \setlength{\parskip}{0pt}
}{\end{enumerate}}

\newenvironment{enumeratenum}{
\begin{enumerate}[{\rm (1)}]
  \setlength{\itemsep}{3.5pt}
  \setlength{\parskip}{0pt}
}{\end{enumerate}}

\CompileMatrices

\renewcommand*{\thefootnote}{\arabic{footnote}}

\numberwithin{equation}{section}

\newtheorem{theorem}{Theorem}[section]
\newtheorem{lemma}[theorem]{Lemma}
\newtheorem{proposition}[theorem]{Proposition}
\newtheorem{corollary}[theorem]{Corollary} 
\newtheorem{definition}[theorem]{Definition}
\newtheorem{example}[theorem]{Example}
\newtheorem*{remark}{Remark}

\newtheorem*{acknowledgements}{Acknowledgements}

\newtheorem*{thmA}{Theorem A}
\newtheorem*{corA}{Corollary A}

\newtheorem*{thmB}{Theorem B}
\newtheorem*{thmC}{Theorem C}
\newtheorem*{thmD}{Theorem D}

\newtheorem*{thm4.1}{Theorem 4.1}
\newtheorem*{thm4.2}{Theorem 4.2}
\newtheorem*{thm4.6}{Theorem 4.6}
\newtheorem*{prop4.7}{Proposition 4.7}
\newtheorem*{prop4.8}{Proposition 4.8}
\newtheorem*{prop4.9}{Proposition 4.9}
\newtheorem*{thm4.10}{Theorem 4.10}

\newtheorem*{prop3.4'}{Proposition 3.4$'$}
\newtheorem*{propA'}{Proposition A$'$}

\newtheorem{openquestion2}[theorem]{Open Question}

\newtheorem{question}[theorem]{Question}

\usepackage{titlesec}

\titleformat*{\section}{\large \bfseries}
\titleformat*{\subsection}{\bf}
\titleformat*{\subsubsection}{\large\bfseries}
\titleformat*{\paragraph}{\large\bfseries}
\titleformat*{\subparagraph}{\large\bfseries}

\begin{document}

\title{Generating numbers of rings graded by amenable and supramenable groups} 
 
\author{
{\sc Karl Lorensen}\\
{\sc Johan \"Oinert}
}

\maketitle

\begin{abstract} A ring $R$ has {\it unbounded generating number} (UGN) if, for every positive integer $n$, there is no $R$-module epimorphism $R^n\to R^{n+1}$. For a ring $R=\bigoplus_{g\in G} R_g$  graded by a group $G$ such that  the base ring $R_1$ has UGN, we identify several sets of conditions under which $R$ must also have UGN. The most important of these are: (1) $G$ is amenable, and there is a positive integer $r$ such that, for every $g\in G$, $R_g\cong (R_1)^i$ as $R_1$-modules for some $i=1,\dots,r$; (2) $G$ is supramenable, and there is a positive integer $r$ such that, for every $g\in G$, $R_g\cong (R_1)^i$ as $R_1$-modules for some $i=0,\dots,r$.
The pair of conditions (1) leads to three different ring-theoretic characterizations of the property of amenability for groups.

We also consider rings that do not have UGN; for such a ring $R$, the smallest positive integer $n$ such that there is an $R$-module epimorphism $R^n\to R^{n+1}$ is called the {\it generating number} of $R$, denoted ${\rm gn}(R)$. If $R$ has UGN, then we define ${\rm gn}(R):=\aleph_0$. We describe several classes of examples of a ring $R$ graded by an amenable group $G$ such that ${\rm gn}(R)\neq {\rm gn}(R_1)$. 
\vspace{10pt}

\vspace{5pt}

\end{abstract}
\let\thefootnote\relax

\newpage
\tableofcontents

\section{Introduction}

This paper relates the group-theoretic properties of amenability and supramenability to the property of having unbounded generating number for rings. 
Following P. M. Cohn \cite{CohnSkew, CohnFIR}, we say that a ring $R$ has \emph{unbounded generating number} (\emph{UGN}) if, for every positive integer $n$, there is no $R$-module epimorphism $R^n\to R^{n+1}$. This property is also referred to as the {\it rank condition} (see, for instance, \cite{Lam}). However, the term {\it UGN} is particularly apposite since the property is equivalent to the assertion that, for every positive integer $n$, there is a finitely generated $R$-module that cannot be generated by fewer than $n$ elements (see Proposition~2.2((i)$\Longleftrightarrow$(iv)) below). 

We study the UGN property for rings that are graded by a group $G$. Recall that a ring $R$ is {\it graded} by $G$, or {\it $G$-graded}, if there is a collection $\{R_g : g\in G\}$ of additive subgroups of $R$ such that $R = \bigoplus_{g\in G} R_g$ as additive groups and $R_gR_h\subseteq R_{gh}$ for all $g,h\in G$. For such a ring, each $R_g$ is referred to as a {\it homogeneous component} of $R$. Moreover, $R_1$ is a subring of $R$, called its {\it base ring}.  The \emph{support} of $R$, denoted ${\rm Supp}(R)$, is defined by ${\rm Supp}(R):=\{g\in G:R_g\neq 0\}$. If $R_gR_h=R_{gh}$ for all $g,h\in G$, then $R$ is said to be \emph{strongly graded} by $G$. 

If each homogeneous component of a $G$-graded ring $R$ contains a unit, then $R$ is called a \emph{crossed product} of $G$ over $R_1$, denoted $R_1\ast G$. Crossed products are plainly strongly graded, and their homogeneous components are all free right and left $R_1$-modules of rank one. An important special case of a crossed product is a \emph{skew group ring} of $G$ over $R_1$, defined by the presence of a group monomorphism $\phi:G\to R^\ast$ such that $\phi(g)\in R_g$ for all $g\in G$. ($R^\ast$ denotes the multiplicative group of units of $R$.) 

If a ring $R$ admits a unital ring homomorphism from $R$ to a UGN-ring, then $R$ must have UGN (see Lemma 2.7 below). As a result, the base ring always inherits the UGN property from a graded ring. Our primary goal in this paper is to investigate the extent to which the converse may hold. In other words, we explore the following question, which was posed by P. ~H.~ Kropholler in an e-mail message to the first author in May of 2020. 
 
\begin{question}[{Kropholler}] Given a group $G$ and a $G$-graded ring $R$, what assumptions concerning $R$ and $G$ will ensure that $R$ must have UGN if $R_1$ has UGN?
\end{question}

One obvious answer to this question is the assumption that $R$ is the group ring of $G$ over $R_1$, for, in this case, the augmentation map $R\to R_1$ is a unital ring homomorphism. In search of other possible answers, we identify a relevant measure-theoretic property of the support of $R$. This property, first studied by A. Tarski \cite{Tar} in 1929 and four and a half decades later by J. M. Rosenblatt \cite{Rosenblatt1, Rosenblatt2}, is defined as follows.

\begin{definition} {\rm Let $G$ be a group and $X$ a subset of $G$. We say that $X$ is \emph{amenable with respect to $G$}, or $X$ is an \emph{amenable subset} of $G$, if there is a function $\mu:\mathcal{P}(G)\to [0,\infty]$ with the following three properties.
\begin{enumerate*}
\item $\mu(A\cup B)=\mu(A)+\mu(B)$ for any sets $A, B\subseteq G$ with $A\cap B=\emptyset$. 
\item $\mu(gA)=\mu(A)$ for any set $A\subseteq G$ and $g\in G$.
\item$\mu(X)=1$.
\end{enumerate*}}
\end{definition}

Our use of the adjective ``amenable" to refer to a subset of a group with this property is a novel one. Nevertheless, the authors regard it as an appropriate choice of word because the concept is a generalization of the usual notion of an amenable group. Specifically, a group $G$ is amenable if and only if the set $G$ is amenable with respect to $G$ in the sense of Definition~1.2. Tarski \cite{Tar} proved that the amenability of a subset $X$ of a group $G$ is equivalent to the nonexistence of a paradoxical decomposition of $X$ with respect to $G$ (see \S 2.4 below). Moreover, Rosenblatt \cite{Rosenblatt1} characterized amenable subsets using a generalization of E. F\o lner's \cite{Folner} well-known condition characterizing amenable groups. 

We employ a modified version (see Definition 2.20 below) of Rosenblatt's characterization to prove our main result, Theorem A, about graded rings with amenable support.

\begin{thmA} Let $G$ be a group and $R$ a ring graded by $G$ such that the following two conditions hold.
\begin{enumerate*} 
\item There is a positive integer $r$ such that, for every $g\in G$,  $R_g\cong (R_1)^i$ as right $R_1$-modules for some $i=0,\dots, r$.
\item The support of $R$ is amenable with respect to $G$.
\end{enumerate*}
Then $R$ has UGN if and only if $R_1$ has UGN. 
\end{thmA}

Note that, in the statement of Theorem A and throughout the paper, we interpret $S^0$ to be the zero $S$-module (left or right, depending on context) for any ring $S$. We point out that there is also a dual version of Theorem A involving left $R_1$-modules and \emph{right amenable} subsets (see Corollary 3.1 below, as well as the paragraph preceding it). 

Theorem A leads directly to our next result, Theorem B, which treats rings graded by amenable groups and provides three ring-theoretic characterizations of the property of amenability for groups.
In the statement of this result and subsequently, we will employ several terms specific to this paper that pertain to a ring $R$ graded by a group $G$. We say that the grading is \emph{full} if ${\rm Supp}(R)=G$. Moreover, we
will refer to the grading as \emph{free} (respectively, \emph{projective}) if the homogeneous components are all finitely generated free (respectively, projective) right $R_1$-modules or all finitely generated free (respectively, projective) left $R_1$-modules.

We call the grading \emph{boundedly free} if there exists a positive integer $r$ such that at least one of the following two conditions is satisfied.
\begin{itemize}
\item For each $g\in G$, there is an integer $i\in [0,r]$ such that $R_g\cong \left (R_1\right )^i$ as right $R_1$-modules.
\item For each $g\in G$, there is an integer $i\in [0,r]$ such that $R_g\cong \left (R_1\right )^i$ as left $R_1$-modules.
\end{itemize}
If the grading is free but not boundedly free, then it is said to be \emph{unboundedly free}. 

The grading is referred to as \emph{boundedly projective} if there is a positive integer $r$ such that at least one of the following two statements is true.
\begin{itemize}
\item For each $g\in G$, $R_g$ is a direct summand in $\left (R_1\right )^r$ as a right $R_1$-module.
\item For each $g\in G$, $R_g$ is a direct summand in $\left (R_1\right )^r$ as a left $R_1$-module.
\end{itemize}
A projective grading is described as \emph{unboundedly projective} if it is not boundedly projective.  

It is important to bear in mind that crossed products are fully graded and boundedly free. Furthermore, every strong grading is necessarily both full and projective (see \cite[Theorem 3.1.1]{NO}). 

\begin{thmB} The following four statements are equivalent for a group $G$.
\begin{enumerate*}
\item G is amenable.
\item Every ring with a full and boundedly free $G$-grading has UGN if and only if its base ring has UGN.
\item Every crossed product of $G$ over a ring $R$ has UGN if and only if $R$ has UGN.
\item Every skew group ring of $G$ over a ring $R$ has UGN if and only if $R$ has UGN. 
\end{enumerate*}
\end{thmB}

Crossed products involving amenable groups form a prominent collection of rings to which we can apply Theorem B to ascertain whether they have UGN. But there are quite a few instances of rings graded by amenable groups that are not crossed products where the homogeneous components are still free modules of rank one over the base ring. One such family are the Weyl rings, treated in Corollary~3.9 below.  Both the Weyl rings and crossed products are special cases of what are called {\it crystalline} and {\it precrystalline} graded rings, investigated in \cite{NO2} and \cite{Oeinert}. Such rings all possess homogeneous components that are free of rank one and thus can be analyzed profitably using Theorem~B, provided that the grading group is amenable. However, the authors are not aware of any interesting applications of Theorem B to rings where the ranks of some of the homogeneous components are larger than one. 

Theorem A also immediately gives rise to a corollary about rings graded by a supramenable group, a type of group that was introduced by Rosenblatt \cite{Rosenblatt2}. He defined a group $G$ to be \emph{supramenable} if every nonempty subset of $G$ is amenable. Significantly, he proved that any group whose finitely generated subgroups all display a subexponential rate of growth must be supramenable (\cite[Theorem~4.6]{Rosenblatt2}). Since finitely generated, virtually nilpotent groups have a polynomial rate of growth, Rosenblatt's result implies that every locally virtually nilpotent group is supramenable. In fact, within the class of elementary amenable groups, these two properties are equivalent (see \cite[Theorem~3.2$'$]{Chou} or \cite[p. 288]{Wagon}). But they fail to be equivalent in general, as demonstrated by the existence of finitely generated groups exhibiting a growth rate that is subexponential but not polynomial (see \cite{Grigor}).

 Our corollary about rings graded by supramenable groups pertains to boundedly free gradings that may not be full.

\begin{corA} Let $G$ be a supramenable group and $R$ a ring with a boundedly free $G$-grading. Then $R$ has UGN if and only if $R_1$ has UGN.  
\end{corA}

Whether such a restrictive condition on the group is really necessary in Corollary A remains unresolved.

\begin{openquestion2} Let $G$ be an amenable group, and let $R$ be a ring that is equipped with a boundedly free $G$-grading. If $R_1$ has UGN, does it necessarily follow that $R$ has UGN?
\end{openquestion2}
 
Turning to projective gradings, we are confronted with a different pattern of behavior than in the free case. Specifically, we show, in Theorem 4.6 below, that there are non-UGN-rings with boundedly projective, full gradings by amenable groups and base rings with UGN. However, it is not known whether such examples exist for supramenable groups. Only in the limited case of a locally finite group can we obtain a definitive result about projective gradings; this is accomplished in Theorem~C.  

\begin{thmC} Let $R$ be a ring that is projectively graded by a locally finite group $G$. Then $R$ has UGN if and only if $R_1$ has UGN.
\end{thmC} 

As shown in \S 3.1, Theorem C follows easily from the fact that the UGN property is a Morita invariant (see Corollary~2.3 below). 

The proofs of Theorem~B((iv)$\implies$(i)) and Theorem~4.6 employ the notion of a 
\emph{translation ring}, a generalization of a concept that was introduced by M.~Gromov \cite{Gromov} in 1993. If $R$ is a ring, $G$ a group, and $X$ a nonempty subset of $G$, then the \emph{translation ring} of $X$ with respect to $G$ over $R$, denoted $T_G(X,R)$, is the ring of all $X\times X$ matrices $M$ over $R$ for which there is a finite set $K\subseteq G$ such that $M(x,y)=0$ whenever $y\notin Kx$. Furthermore, we define $T_G(\emptyset, R)$ to be the zero ring for any group $G$ and ring $R$. The ring $T_G(G,R)$ is denoted simply by $T(G,R)$. For a finitely generated group $G$, the rings $T(G,R)$ are the translation rings introduced by Gromov and investigated further in \cite{Ara1}, \cite{Ara2}, \cite{Elek}, and \cite{Roe}. 

 Translation rings are important examples of group-graded rings; in the case where $X=G$, they are skew group rings (see Proposition~2.14 below), but, otherwise, they may be neither freely nor strongly graded (see the proof of Proposition~4.7 below). Our interest in translation rings stems primarily from the fact that they can be used to characterize amenable subsets of groups, as we show in Theorem~D. 

\begin{thmD} Let $G$ be a group and $X$ a subset of $G$. Then the following two statements are equivalent.

\begin{enumerate*}
\item The subset $X$ is amenable with respect to $G$.
\item For every ring $R$, the ring $T_G(X,R)$ has UGN if and only if $R$ has UGN. 
\end{enumerate*}
\end{thmD}

The implication (ii)$\implies$(i) in Theorem D for the case $X=G$ and $G$ finitely generated  was already shown by G. Elek \cite[\S 2]{Elek}. 
Our reasoning for that direction is an extension of his and relies on paradoxical decompositions of subsets of groups. We prove the converse with an argument similar to that employed for Theorem~A,  invoking Rosenblatt's generalization of F\o lner's condition. 
Theorem~D has the benefit of furnishing characterizations of both amenability and supramenability in terms of translation rings: a group $G$ is amenable if and only if, for every UGN-ring $R$, the ring $T(G,R)$ has UGN (Corollary 3.5); a group $G$ is supramenable if and only if, for every nonempty set $X\subseteq G$ and UGN-ring $R$, the ring $T_G(X,R)$ has UGN
(Corollary 3.6). 

In the final section of the paper, we investigate rings that fail to have UGN; such rings are said to have {\it bounded generating number} (\emph{BGN}). If a ring $R$ has BGN, then the smallest integer $n>0$ for which there is an $R$-module epimorphism $R^n\to R^{n+1}$ is called the {\it generating number} of $R$, denoted  ${\rm gn}(R)$. The generating number of a BGN-ring can also be characterized as the smallest integer $n>0$ such that every finitely generated $R$-module can be generated by $n$ elements (see Proposition~2.4((i)$\Longleftrightarrow$(iv)) below). 
If $R$ has UGN, then we define ${\rm gn}(R):=\aleph_0$. Thus, for an arbitrary ring $R$, the generating number of $R$ is the smallest cardinal number $c$ such that every finitely generated $R$-module can be generated by a subset of cardinality~$c$. 

As far as the authors are aware, the notion of the generating number of a ring has never been studied explicitly before. In the finite case, it is analogous to the concept of 
the type of a non-IBN-ring (see \cite{Leavitt}), and many of the same techniques that have been used to investigate types can be applied to the study of finite generating numbers. This is particularly true of the arguments adduced in \cite{Abrams} and \cite{Loc}; indeed, some of our results on finite generating numbers are proved using methods from those two papers. 

Our results about finite generating numbers are all of a negative nature, showing that Theorem~B, Corollary~A, and Theorem~C cannot be generalized in specific fashions.
The most significant one, Theorem~4.1, demonstrates that Theorem~B((i)$\implies$(ii)) and Corollary~A fail to extend to rings that have unboundedly free gradings. 

\begin{thm4.1} There exists a ring $R$ with a full and unboundedly free $\mathbb Z$-grading such that $R_0$ has UGN and $R$ has BGN.
\end{thm4.1}

In the statement of Theorem 4.1, notice that, since $\mathbb Z$ is an additive group with identity element $0$, the base ring of $R$ is written $R_0$. 

The ring described in the proof of Theorem 4.1 is not strongly graded. In our next theorem, we show that we can construct strongly $\mathbb Z$-graded examples that exhibit the same dissonance between the base ring and the entire ring. 

\begin{thm4.2} For any positive integer $n$, there exists a strongly $\mathbb Z$-graded ring $R$ such that $R_0$ has UGN and ${\rm gn}(R)=n$. 
\end{thm4.2}

The example given for Theorem 4.2 has a grading that is unboundedly projective and not free. In Theorem 4.6, we construct an example illustrating the same phenomenon that has a boundedly projective grading. This grading, however, is not strong.

\begin{thm4.6}  Let $G$ be an arbitrary elementary amenable group that fails to be locally virtually nilpotent. For every $n\in \mathbb Z^+$, there exists a ring $R$ graded by $G$ with the following properties.
\begin{enumerate*}
\item The $G$-grading on $R$ is full. 
\item For each $g\in G$, the $R_1$-module $R_g$ is a direct summand in $R_1$. 
\item $R_1$ has UGN.
\item ${\rm gn}(R)=n$.
\end{enumerate*}
\end{thm4.6}

Corollary A  says that, for a ring $R$ equipped with a boundedly free grading by a supramenable group $G$, we have 
\[{\rm gn}(R_1)=\aleph_0\Longleftrightarrow {\rm gn}(R)=\aleph_0.\]
One might, therefore, ask whether these two generating numbers always coincide for such rings. The answer, however, is negative, as shown by our next result.

\begin{prop4.8} Let $m, n$ be positive integers with $m\leq n$, and let $G$ be a group of order $n$.
Then there are a ring $R$ and a skew group ring $R\ast G$ such that ${\rm gn}(R)=n$ and ${\rm gn}(R\ast G)=m$. 
\end{prop4.8} 

Proposition 4.8 also serves to show that Theorem~C fails to extend to finite generating numbers. In Proposition 4.9 below, we prove the existence of another form of example confirming this fact for an arbitrary pair of positive integers. Unlike the ring $R\ast G$ described in Proposition 4.8, the one in Proposition~4.9 is graded by an arbitrary nontrivial finite group, although it is not possible in this case to make the ring a crossed product.  

\begin{prop4.9} Let $G$ be a nontrivial finite group and $m, n$ integers with $0<m\leq n$. Then there exists a ring $R$ strongly graded by $G$ such that
${\rm gn}(R_1)=n$ and ${\rm gn}(R)=m$. 
\end{prop4.9}

We conclude the paper by showing that the hypothesis that the grading is projective in Theorem C cannot be dropped. 

\begin{thm4.10} Let $G$ be a nontrivial finite group. Then, for any integer $n>0$, there exists a $G$-graded ring $R$ such that
$R_1$ has UGN and ${\rm gn}(R)=n$.
\end{thm4.10}

\begin{acknowledgements} {\rm The authors wish to thank Peter Kropholler for posing the question that became the subject of this paper, as well as for pointing out the possible relevance of Morita theory in this context. In addition, we are grateful to Nicolas Monod for suggesting Example 3.2.  Moreover, the first author thanks Blekinge Tekniska H\"ogskola in Karlskrona, Sweden for generously hosting him for part of the time during which the paper was written.

Finally, the authors wish to express their gratitude to the anonymous mathematicians who evaluated the paper on behalf of the journal for providing many salutary suggestions. In particular, one of the referees detected a subtle error in the proof of the main result in an earlier version of the paper, which led the authors to recognize the need for the hypothesis that the grading be bounded. 
 }
\end{acknowledgements}

 \section{Notation and preliminary results}

In this section, we present some elementary facts about UGN, generating numbers, and translation rings. In addition, we discuss amenable subsets of groups, particularly their descriptions in terms of paradoxical decompositions and the F\o lner condition for subsets of groups. We begin by describing the notation and terminology that we employ throughout the paper.
\newpage

\subsection{Notation and terminology}
\vspace{10pt}

\noindent {\it General notation.}
\vspace{5pt}

\noindent If $X$ is a set, then $\mathcal{P}(X)$ is the power set of $X$ and $|X|$ the cardinality of $X$.
\vspace{5pt}

\noindent $\mathbb Z^+$ is the set of positive integers.
\vspace{5pt}

\noindent  $\mathbb R^+$ is the set of positive real numbers.
\vspace{5pt}

\noindent For any objects $a, b$, 

\[\delta(a,b):=\begin{cases} 1 &\ \mbox{if}\ a=b\\
0 &\ \mbox{if}\ a\neq b. \end{cases}\]
For $i, j\in \mathbb Z^+$, we write $\delta_{ij}:=\delta(i,j)$.
\vspace{10pt}

\noindent {\it Groups, rings, and modules.}
\vspace{5pt}

\noindent The identity element of a group will be denoted by $1$.
\vspace{5pt}

\noindent The class of \emph{elementary amenable} groups is the smallest class of groups that contains all finite groups and all abelian groups and that is closed under forming extensions and direct limits. Note that every elementary amenable group is amenable (see \cite[Propositions 4.4.6, 4.5.5, 4.5.10, 4.6.1]{CSC}). 
\vspace{5pt}

\noindent All rings, subrings, and ring homomorphisms are assumed to be unital.
\vspace{5pt}

\noindent Let $R$ be a ring. Then $R^\ast$ denotes the group of units (invertible elements) in $R$, and $Z(R)$ is the center of $R$. Moreover, $R^{\rm op}$ represents the ring opposite to $R$, that is, the ring with the same set of elements as $R$ and same addition, but whose multiplication $\circ$ is defined by $r\circ s:=sr$ for all $r, s\in R^{\rm op}$. 
\vspace{5pt}

\noindent If $R$ is a ring and $G$ a group, then $RG$ denotes the group ring of $G$ over $R$. 
\vspace{5pt}

\noindent The term {\it module} without the modifier {\it left} will always mean right module. 
\vspace{5pt}

\noindent If $R$ is a ring, then $\mathfrak M_R$ is the category of $R$-modules. 
\vspace{5pt}

\noindent Let $R$ be a ring graded by a group $G$ and $r\in R$. If $r=\sum_{g\in G} r_g$ with $r_g\in R_g$ for all $g\in G$, then the {\it support} of $r$, denoted ${\rm Supp}(r)$, is defined by
\[{\rm Supp}(r):=\{g\in G\ :\ r_g\neq 0\}.\]
\vspace{5pt}

\noindent Let $S$ be a ring. Following G. M. Bergman \cite{Bergman, diamond}, we use the term \emph{$S$-ring} for a ring $R$ that is also an $S$-$S$-bimodule such that the following three conditions are satisfied for all $r, r'\in R$ and $s\in S$:
\begin{enumerate*}
\item $(rr')s=r(r's)$;
\item $(sr)r'=s(rr')$;
\item $(rs)r'=r(sr')$.
\end{enumerate*}

If $R$ and $R'$ are both $S$-rings, then an \emph{$S$-ring homomorphism} from $R$ to $R'$ is a ring homomorphism $R\to R'$ that is also an $S$-$S$-bimodule homomorphism. 

Note that, if $S$ is a commutative ring, then an \emph{$S$-algebra}, in the traditional sense, is an $S$-ring $R$ such that $sr=rs$ for all $s\in S$ and $r\in R$.  
\vspace{10pt}

\noindent {\it Crossed products, skew group rings, and twisted group rings.}
\vspace{5pt}

Let $G$ be a group and $R$ a ring. Let $\sigma:G\to {\rm Aut}(R)$ and $\omega:G\times G\to R^\ast$ be functions. Write $g\cdot r:=(\sigma(g))(r)$ for $g\in G$, $r\in R$. 
If $\sigma$ is a group homomorphism, then we call $\sigma$ an {\it action} of $G$ on $R$. If $\sigma$ is the trivial group homomorphism, it is referred to as the {\it trivial action}. 

The quadruple $(G,R,\sigma, \omega)$ is called a {\it crossed system} if the maps $\sigma$ and $\omega$ satisfy the following three conditions for any $g,h,k\in G$ and $r\in R$:
\begin{enumerate*}
\item $g\cdot (h\cdot r)=\omega(g,h)\, ((gh)\cdot r)\, (\omega(g,h))^{-1}$;
\item $\omega(g,h)\, \omega(gh,k) = (g\cdot \omega(h,k))\, \omega(g,hk)$;
\item $\omega(g,1_G)=\omega(1_G,g)=1_R$.
\end{enumerate*}

If $(G,R,\sigma, \omega)$ is a crossed system, then the set of formal sums $\sum_{g\in G} r_g\, g$, in which $r_g$ is an element of $R$ for all $g\in G$ and zero for all but finitely many $g\in G$, can be made into a ring by defining addition componentwise and multiplication according to the rule
\[(r_g\, g)(r_h\, h):=r_g(g\cdot r_h)\, \omega(g,h)\, (gh)\]

\noindent for all $g, h\in G$ and $r_g, r_h\in R$ (see \cite[Proposition 1.4.1]{NO}). This ring is a crossed product of $G$ over $R$ with respect to the obvious $G$-grading; it is denoted $R\ast^\sigma_\omega G$.
Moreover, every crossed product of $G$ over $R$ is isomorphic to a ring of this form (see \cite[Proposition 1.4.2]{NO}). 

An important special case of a crossed product arises when $\omega$ is the trivial map, meaning $\omega(g,h)=1_R$ for all $g, h\in G$. In this case, $(G,R,\sigma, \omega)$ is a crossed system if and only if $\sigma$ is an action of $G$ on $R$.
When $\omega$ is trivial and $\sigma$ an action, the crossed product $R\ast^\sigma_\omega G$ is referred to as the {\it skew group ring} of $G$ over $R$ that is associated to the action $\sigma$. 

A second important case is where $\sigma$ is the trivial action. In this case, a map $\omega:G\times G\to R^\ast$ satisfies (i) if and only if ${\rm Im}\ \omega\subseteq Z(R)^\ast$. Moreover, in the presence of this containment, the map $\omega$ fulfills (ii) if and only if $\omega$ is a $2$-cocycle $G\times G\to Z(R)^\ast$. If, in addition, (iii) holds, then we call $\omega$ a {\it normalized $2$-cocycle}. If $\sigma$ is the trivial action and $\omega$ a normalized $2$-cocycle $G\times G\to Z(R)^\ast$, then $R\ast^\sigma_\omega G$ is called the {\it twisted group ring} of $G$ over $R$ that is associated to the normalized $2$-cocycle $\omega$. Furthermore, the twisted group ring $R\ast^\sigma_\omega G$ is isomorphic to the group ring $RG$ if and only if the cohomology class of $\omega$ in $H^2(G,Z(R)^\ast)$ is zero (see \cite[\S 1.5, Exercise 10]{NO}).  
\vspace{10pt}

\noindent {\it Monoids.}
\vspace{5pt}

\noindent For any ring $R$, ${\rm Proj}(R)$ is the abelian monoid consisting of the isomorphism classes of finitely generated projective $R$-modules under the operation
 of forming direct sums.
\vspace{5pt}

\noindent  For an abelian monoid $M$, we write $x\leq y$ for $x, y\in M$ if there exists $z\in M$ such that $x+z=y$.
\vspace{5pt}

\noindent For any two positive integers $n$ and $k$, $C(n,k)$ denotes the monoid with presentation
\[C(n,k):=\langle a\ :\ (n+k)a=na\rangle.\]
\vspace{10pt}

\noindent {\it Matrices.}
\vspace{5pt}

\noindent If $X$ and $Y$ are nonempty sets, an $X\times X$ {\it matrix} $M$ with entries in $Y$ (or over $Y$) is a function $M:X\times X\to Y$. For an $m\times n$ matrix $M$ with $m, n\in \mathbb Z^+$, we write $M_{ij}:=M(i,j)$ for $i=1,\dots, m$ and $j=1,\dots, n$. 
\vspace{5pt}

\noindent The ring of $n \times n$ matrices with entries in a ring $R$ will be denoted $M_n(R)$.
\vspace{5pt}

\noindent The transpose of a (finite or infinite) matrix $A$ is denoted $A^t$.
\vspace{5pt}

\noindent Let $G$ be a group, $R$ a ring, and $X$ a nonempty subset of $G$. If $M, N\in T_G(X,R)$, then the product $MN$ is defined by
\[(MN)(x,y):=\sum_{z\in X} M(x,z)N(z,y)\]
for $x,y\in X$. Notice that this sum is defined because it has only finitely many nonzero terms.

\subsection{UGN and generating numbers}

 In this subsection, we introduce the reader to some elementary facts about the UGN property and generating numbers. The most obvious examples of UGN-rings are nonzero division rings and nonzero finite rings. Other examples may be obtained by using the fact that every ring $R$ admitting a ring homomorphism from $R$ to a UGN-ring must have UGN (see Lemma 2.7). In particular, every nonzero commutative ring has UGN since its quotient by a maximal ideal is a field.  
 
 The UGN property was introduced by Cohn \cite{CohnIBN} in connection with the related, but currently much better known, properties of stable finiteness and having invariant basis number. A ring $R$ is said to be \emph{stably finite} (\emph{SF}) if, for every positive integer $n$, every $R$-module epimorphism $R^n\to R^n$ must be an isomorphism. A ring $R$ has \emph{invariant basis number} (\emph{IBN}) if, for any two positive integers $m$ and $n$, $R^m\cong R^n$ as $R$-modules only if $m=n$. It is easy to see that the following implications hold for any nonzero ring:
 \[\mbox{SF}\ \Longrightarrow\ \mbox{UGN}\ \Longrightarrow \mbox{IBN}.\]
 Furthermore, neither of these two implications can be reversed: for the first, see, for instance, \cite[p. 11]{Lam}; for the second, see \cite[\S 5]{CohnIBN} and \cite[Example 3.19]{AbramsNP}. The precise relationship between SF and UGN is made clear by P. Malcolmson's result \cite{Malcolm} that a ring has UGN if and only if it has a nonzero SF quotient
(see also \cite[Theorem 1.26]{Lam}). 
 
 Crucial to our understanding of generating numbers of rings is the following lemma.  

\begin{lemma} Let $R$ be a ring, $A$ an $R$-module, and $n\in \mathbb Z^+$. If there is an $R$-module epimorphism $A^n\to A^{n+1}$, then there is an $R$-module epimorphism $A^n\to A^m$ for every $m\in \mathbb Z^+$.
\end{lemma}

\begin{proof} Let $\phi:A^n\to A^{n+1}$ be an $R$-module epimorphism. We will show by induction on $k$ that there is an $R$-module epimorphism $A^n\to A^{n+k}$ for every $k\in \mathbb Z^+$. Since there is clearly an $R$-module epimorphism $A^n\to A^m$ when $m\leq n$, this will prove the lemma. 
Suppose that there is an $R$-module epimorphism $\psi:A^n\to A^{n+k-1}$ for $k>1$. Then we have the chain of $R$-module epimorphisms
\[A^n\stackrel{\psi}{\longrightarrow} A^{n+k-1}\stackrel{\cong}{\longrightarrow} A^n\oplus A^{k-1}\stackrel{\xi}{\longrightarrow} A^{n+1}\oplus A^{k-1}\stackrel{\cong}{\longrightarrow} A^{n+k},\]
where $\xi(a,b):=(\phi(a),b)$ for all $a\in A^n$ and $b\in A^{k-1}$. The composition of the maps in the chain is an $R$-module epimorphism $A^n\to A^{n+k}$. 
\end{proof}

Lemma 2.1 allows us to establish several alternative characterizations of rings with BGN. These are all well known and can be found, stated less formally, in the works of Cohn (see, for instance, \cite[\S 1.4]{CohnSkew}). 

\begin{proposition} For a ring $R$, the following five statements are equivalent.
\begin{enumerate*}
\item $R$ has BGN.
\item There is an $R$-module epimorphism $R^n\to R^m$ for some $m, n\in \mathbb Z^+$ with $n<m$. 
\item There is an integer $n>0$ such that there is an $R$-module epimorphism $R^n\to R^m$ for every $m\in \mathbb Z^+$. 
\item There is an integer $n>0$ such that every finitely generated $R$-module is a homomorphic image of $R^n$.
\item There is a finitely generated $R$-module $M$ such that every finitely generated $R$-module is a homomorphic image of $M$.
\end{enumerate*}
\end{proposition}

\begin{proof} The equivalence of (iv) and (v) is trivial, as are the implications (i)$\implies$(ii), (iii)$\implies$(iv), and (iv)$\implies$(i). Moreover, the implication (ii)$\implies$(iii) follows from Lemma 2.1.
\end{proof}

Statement (v) in Proposition 2.2 is significant in that it is expressed in purely categorical terms. As a consequence, its equivalence to (i) implies that BGN, and perforce UGN, is Morita invariant, thus furnishing a succinct, positive answer to a question in an exercise by Cohn \cite[Exercise 9, \S 0.1]{CohnFIR}. For the benefit of the reader, we provide all the details of this argument in the proof of Corollary 2.3 below. An alternative proof that focuses instead on the structure of the monoid ${\rm Proj}(R)$ for a ring $R$ was constructed by P. Ara and appears in \cite[\S 2]{AbramsNP}. 

\begin{corollary} Let $R$ and $S$ be rings that are Morita equivalent. Then $R$ has UGN if and only if $S$ has UGN. 
\end{corollary}

 \begin{proof} The argument is based on the fact that category equivalences preserve both the property of surjectivity for homomorphisms and that of finite generation for modules (see, for example, \cite[Proposition 21.2]{Anderson-Fuller} and \cite[Proposition 21.8]{Anderson-Fuller}, respectively). Notice that it suffices to prove that, if $R$ has BGN, then $S$ has BGN.  Assume that $R$ has BGN.  According to Proposition~2.2((i)$\implies$(v)), there is a finitely generated $R$-module $M$ such that every finitely generated $R$-module is a homomorphic image of $M$. Let $F:\mathfrak{M}_R\to \mathfrak{M}_S$ be a category equivalence. 
 This means that there is a functor $G:\mathfrak{M}_S\to \mathfrak{M}_R$ such that $FG$ and $GF$ are naturally equivalent to the identity functors $\mathfrak{M}_S\to \mathfrak{M}_S$ and $\mathfrak{M}_R\to \mathfrak{M}_R$, respectively. Furthermore, the $S$-module $F(M)$ is finitely generated. 
 
 We will now show that every finitely generated $S$-module is a homomorphic image of $F(M)$. This will imply, by Proposition~2.2((v)$\implies$(i)), that $S$ has BGN. 
 Let $A$ be a finitely generated $S$-module. Since $G(A)$ is a finitely generated $R$-module, we have an $R$-module epimorphism $\phi:M\to G(A)$. Thus the map $F(\phi):F(M)\to FG(A)$ is an $S$-module epimorphism. But $FG(A)\cong A$ as $S$-modules, yielding an $S$-module epimorphism $F(M)\to A$. 
\end{proof}

\begin{remark}{\rm The authors have succeeded in constructing two other short proofs that UGN is Morita invariant. A second approach is to first establish the following characterization of UGN: a ring $R$ has UGN if and only if, for every progenerator $P$ in $\mathfrak{M}_R$ and integer $n>0$, there is no $R$-module epimorphism $P^n\to P^{n+1}$. The Morita invariance of UGN then follows from the fact that category equivalences preserve progenerators and epimorphisms.}
\end{remark}

Reasoning similar to that employed in proving Proposition 2.2 can be invoked to establish the following set of equivalences concerning the generating number of a ring.

\begin{proposition} Let $R$ be a ring and $n\in \mathbb Z^+$. Then the following four statements are equivalent. 
\begin{enumerate*}
\item ${\rm gn}(R)=n$.
\item The integer $n$ is the smallest positive integer such that  there is an $R$-module epimorphism $R^n\to R^{m}$ for some integer $m>n$.
\item The integer $n$ is the smallest positive integer such that there is an $R$-module epimorphism $R^n\to R^{m}$ for every $m\in \mathbb Z^+$.
\item The integer $n$ is the smallest positive integer such that every finitely generated $R$-module is a homomorphic image of $R^n$. 
\end{enumerate*}
\end{proposition}

\begin{remark}{\rm We point out that the generating number of a ring fails to be a Morita invariant: see Lemma 2.9 and Proposition 2.12 below.}
\end{remark}

The next lemma is useful in that it allows us to formulate the conditions in Proposition 2.4, as well as (i)-(iii) in Proposition 2.2, in terms of matrices.  Because the property is well known and the proof is very straightforward, we leave it to the reader.
 
\begin{lemma} Let $R$ be a ring and $m, n\in \mathbb Z^+$. Then there is an $R$-module epimorphism $R^n\to R^m$ if and only if there are an $m \times n$ matrix $A$ with entries in $R$ and an $n \times m$ matrix $B$ with entries in $R$ such that $AB=I_m$.
\end{lemma} 

Lemma 2.5 enables us to establish the following four fundamental lemmas about generating numbers (Lemmas 2.6, 2.7, 2.9, and 2.10). The special cases of Lemmas 2.6, 2.7, and 2.10 for infinite generating numbers are all well known (see \cite{CohnIBN}, \cite[\S 0.1]{CohnFIR}, and \cite[\S 1C]{Lam}, respectively). 

\begin{lemma} Let $R$ be a ring. Then ${\rm gn}(R)={\rm gn}(R^{\rm op})$. In particular, $R$ has UGN if and only if $R^{\rm op}$ has UGN. 
\end{lemma}

\begin{proof} If $A$ is an $m\times n$ matrix with entries in $R$ and $B$ an $n\times p$ matrix with entries in $R$, then we employ $A\circ B$ to denote the product of $A$ and $B$ where the entries are multiplied in $R^{\rm op}$. This means 
\begin{equation}A\circ B=(B^tA^t)^t.\end{equation}

Since $(R^{\rm op})^{\rm op}=R$, we only need to show that ${\rm gn}(R^{\rm op})\leq {\rm gn}(R)$. This is plainly true if ${\rm gn}(R)=\aleph_0$, so we assume that $k:={\rm gn}(R)$ is finite. By Proposition 2.4((i)$\implies$(ii)) and Lemma 2.5, there are an integer $l>k$, an $l\times k$ matrix $P$ with entries in $R$, and a $k\times l$ matrix $Q$ over $R$ such that $PQ=I_l$. From (2.1) we obtain $Q^t\circ P^t=(PQ)^t=I_l$. Therefore, appealing to Lemma 2.5 and Proposition 2.4((i)$\implies$(ii)), we deduce ${\rm gn}(R^{\rm op})\leq k$.
\end{proof}

\begin{remark}{\rm Lemma 2.6 means that the UGN and generating number notions are left-right symmetric; in other words, any of their characterizations in terms of right $R$-modules
can also be expressed using left $R$-modules.}
\end{remark}

\begin{lemma} Let $R$ and $S$ be rings such that there is a ring homomorphism $\phi:R\to S$. Then ${\rm gn}(R)\geq {\rm gn}(S)$. In particular, if $S$ has UGN, then so does $R$.
\end{lemma}

\begin{proof} For any $m\times n$ matrix $M$ with entries in $R$, we let $_\phi M$ denote the $m\times n$ matrix with entries in $S$ such that $(_\phi M)_{ij}:=\phi(M_{ij})$ for $i=1,\dots, m$ and $j=1,\dots, n$. The desired inequality obviously holds if ${\rm gn}(R)=\aleph_0$; hence we suppose that ${\rm gn}(R)$ is finite and write $k:={\rm gn}(R)$. By Proposition 2.4((i)$\implies$(ii)) and Lemma~2.5, there are an integer $l>k$, an $l\times k$ matrix $A$ over $R$, and a $k\times l$ matrix $B$ over $R$ such that $AB=I_l$. This implies $(_\phi A)(_\phi B)=I_l$. It follows, then, from Lemma~2.5 and Proposition~2.4((i)$\implies$(ii)) that ${\rm gn}(S)\leq k$. 
\end{proof}

Lemma 2.7 furnishes the following property of group rings. 

\begin{corollary} Let $R$ be a ring and $G$ a group. Then ${\rm gn}(RG)={\rm gn}(R)$. In particular, $RG$ has UGN if and only if $R$ has UGN.
\end{corollary} 

\begin{proof} There is a ring embedding $R\to RG$, and the augmentation map is a ring homomorphism $RG\to R$. Hence the conclusion follows from Lemma 2.7. 
\end{proof}

\begin{lemma} Let $R$ be a ring with BGN, and let $n:={\rm gn}(R)$. If $m$ is a positive integer divisor of $n$, then ${\rm gn}(M_m(R))=\frac{n}{m}$.
\end{lemma}

\begin{proof} 

Because $M_m(R)$ possesses a subring that is isomorphic to $R$, Lemma 2.7 implies that $M_m(R)$ has BGN. Put $p:={\rm gn}(M_m(R))$. Appealing to Proposition~2.4((i)$\implies$(ii)) and Lemma~2.5, we obtain an integer $q~>~p$, a $q\times p$ matrix $A$ with entries in $M_m(R)$, and a $p\times q$ matrix $B$ with entries in $M_m(R)$ such that $AB=I_q$. This gives rise to an $mq\times mp$ matrix $A'$ over $R$ and an $mp\times mq$ matrix $B'$ over $R$ such that $A'B'=I_{mq}$. Consequently, by
Lemma~2.5 and Proposition~2.4((i)$\implies$(ii)), we have $n\leq mp$. 

Set $r:=\frac{n}{m}$. Referring to Proposition 2.4((i)$\implies$(iii)) and Lemma 2.5, we can acquire an integer $k>r$, a $km\times n$ matrix $P$ with entries in $R$, and an $n\times km$ matrix $Q$ with entries in $R$ such that $PQ=I_{km}$. This set-up can then be converted into a $k\times r$ matrix $P'$ with entries in $M_m(R)$ and an $r\times k$ matrix $Q'$ with entries in $M_m(R)$ such that $P'Q'=I_k$. It follows, therefore, from 
Lemma~2.5 and Proposition~2.4((i)$\implies$(ii)) that $p\leq r$. Combining this with the result from the previous paragraph yields $p=r$. 
\end{proof}

\begin{lemma} Let $\{R_{\alpha}:\alpha\in I\}$ be a directed system of rings. If $R:=\varinjlim R_\alpha$, then 
\[{\rm gn}(R)=\min \{{\rm gn}(R_\alpha):\alpha\in I\}.\]

\noindent In particular, $R$ has UGN if and only if $R_\alpha$ has UGN for every $\alpha\in I$. 
\end{lemma} 

\begin{proof} 

Lemma 2.7 implies ${\rm gn}(R)\leq {\rm gn}(R_\alpha)$ for all $\alpha\in I$. The desired conclusion will therefore follow if we show that ${\rm gn}(R_\alpha)\leq {\rm gn}(R)$ for some $\alpha\in I$. Since this statement holds trivially if ${\rm gn}(R)=\aleph_0$, suppose that ${\rm gn}(R)$ is finite and set $p:={\rm gn}(R)$. Invoking Proposition 2.4((i)$\implies$(ii)) and Lemma 2.5, we obtain an integer $q>p$, a $q\times p$ matrix $A$ with entries in $R$, and a $p\times q$ matrix $B$ over $R$ such that $AB=I_q$. This means that, for some $\alpha\in I$, there are a $q\times p$ matrix $A'$ over $R_\alpha$ and a $p\times q$ matrix $B'$ over $R_\alpha$ such that $A'B'=I_q$.
Therefore, appealing to Lemma 2.5 and Proposition 2.4((i)$\implies$(ii)), we conclude that ${\rm gn}(R_\alpha)\leq p$. 
\end{proof}

Our next result describes the relationship between the generating number of a direct product of rings and the generating numbers of the individual factors. The proof is very straightforward; nonetheless, as with the previous four lemmas, we provide the details for the sake of completeness. 

\begin{lemma} Let $\{R_{\alpha}:\alpha\in I\}$ be a family of rings indexed by a set $I$, and write $R:=\prod_{\alpha \in I} R_\alpha$. Then 
\[{\rm gn}(R)=\sup \{{\rm gn}(R_\alpha):\alpha\in I\}.\]
\end{lemma}

\begin{proof} That ${\rm gn}(R)$ is an upper bound for the set $\mathcal S:=\{{\rm gn}(R_\alpha):\alpha\in I\}$ follows by applying Lemma 2.7 to the projection maps $R\to R_\alpha$. Suppose that $b\in \mathbb Z^+\cup \{\aleph_0\}$ is an arbitrary upper bound
for $\mathcal S$. We wish to establish that ${\rm gn}(R)\leq b$. This is plainly true if $b=\aleph_0$, so assume that $b$ is finite.  For each $\alpha\in I$, there is an
$R_\alpha$-module epimorphism $\phi_\alpha:\left (R_\alpha\right )^b\to \left (R_\alpha\right )^{b+1}$. The maps $\phi_\alpha$, then, induce an $R$-module epimorphism $\phi:\prod_{\alpha\in I} \left (R_\alpha\right )^b\to \prod_{\alpha\in I} \left (R_\alpha\right )^{b+1}$. Moreover, since $\prod_{\alpha\in I} \left (R_\alpha\right )^k\cong R^k$ as $R$-modules for every $k\in \mathbb Z^+$, the map $\phi$ induces an $R$-module epimorphism $R^b\to R^{b+1}$. Thus ${\rm gn}(R)\leq b$, and so ${\rm gn}(R)=\sup \mathcal S$. 
\end{proof}

In our final preliminary result about generating numbers, we observe that every positive integer can be realized as the generating number of some ring. 

\begin{proposition} For any integer $n>0$, there is a ring $R$ such that ${\rm gn}(R)=n$.
\end{proposition}

We will deduce Proposition 2.12 from the following important result of Bergman \cite[Theorem 6.2]{Bergman}, which will also serve as the cornerstone of our proof of Theorem~4.10.  

\begin{theorem}{\rm (Bergman)} Let $M$ be a finitely generated abelian monoid with distinguished element $I\neq 0$ such that the following two properties are satisfied.
\begin{enumerate*}
\item For all $x,y\in M$, if $x+y=0$, then $x=y=0$. 
\item For each $x\in M$, there exists $\lambda\in \mathbb Z^+$ such that $x\leq \lambda I$. 
\end{enumerate*}

\noindent Then a ring $R$ exists such that there is a monoid isomorphism $\theta: M\to {\rm Proj}(R)$ with $\theta(I)=[R]$. 
\end{theorem}

\begin{proof}[Proof of Proposition 2.12]  Let $k\in \mathbb Z^+$. The abelian monoid $C(n,k)$ satisfies the hypotheses of Theorem 2.13 where the distinguished element is the generator $a$. Thus there are  a ring $R$ and a monoid isomorphism $\theta:C(n,k)\to {\rm Proj}(R)$ with $\theta(a)=[R]$. Observe that $n$ is the smallest positive integer such that $(n+1)a\leq na$ in the monoid $C(n,k)$. As a result, $n$ is also the smallest positive integer such that $R^{n+1}$ is an $R$-module direct summand  of $R^n$. In other words, ${\rm gn}(R)=n$.
\end{proof}

\begin{remark}  {\rm A concrete example of a ring $R$ having ${\rm Proj}(R)\cong C(n,k)$ is the $K$-algebra $V_{n, n+k}$ defined in \cite{Bergman} for a fixed field $K$ (see \cite[Theorem 6.1]{Bergman}). These algebras were first studied by W. G. Leavitt in his seminal papers \cite{Leavitt3, Leavitt2, Leavitt}. In his honor, the current standard notation for the $K$-algebra $V_{n, n+k}$ is $L_K(n,n+k)$ (see \cite{LPA}). It enjoys the following universal property:  for any $K$-algebra $R$, we have $R^n\cong R^{n+k}$ as $R$-modules if and only if there is a $K$-algebra homomorphism $L_K(n,n+k)\to R$.} 
\end{remark}

\subsection{Translation rings}

Below we prove an important property of translation rings of groups, namely, that they can be viewed as skew group rings. This result is similar to \cite[Theorem 4.28]{Roe} and is probably well known in the case of a finitely generated group, although the authors are not aware of its appearance anywhere in the literature.   

\begin{proposition} For any ring $R$ and group $G$, the ring $T(G,R)$ is isomorphic to the skew group ring $\displaystyle{\prod_G R\ \ast\ G}$, where the action of $G$ on $\displaystyle{\prod_GR}$ is defined by
\[(g\cdot f)(x):=f(g^{-1}x)\ \ \mbox{for}\ \ g, x\in G.\]
\end{proposition} 

\begin{proof}  For an arbitrary element $g$ of $G$, define $A_g$ to be the matrix in $T(G,R)$ such that  $A_g(x,y):=1$ if  $y=g^{-1}x$ and $A_g(x,y):=0$ if $y\neq g^{-1}x$. Moreover, for each function $f:G\to R$, let $D_f$ be the $G\times G$ diagonal matrix with $D_f(x,x):=f(x)$ for all $x\in G$. An easy calculation establishes  
\[\left(A_gD_fA_g^{-1}\right)(x,y)=\begin{cases}f(g^{-1}x) &\mbox{if}\ y=x\\ 0 &\mbox{if}\ y\neq x \end{cases}\ =\ D_{g\cdot f}(x,y)\]
for every $f:G\to R$ and $g\in G$. This means that there is a ring homomorphism $\phi:\prod_G R\ \ast\ G\to T(G,R)$ such that $\phi(f)=D_f$ for every function $f:G\to R$ and $\phi(g)=A_g$ for every $g\in G$. 

Another straightforward calculation shows that, for any function $f:G\to R$ and $g\in G$, 
\[(D_fA_g)(x,y)=\begin{cases} f(x) &\mbox{if}\ y=g^{-1}x\\ 0 &\mbox{if}\ y\neq g^{-1}x \end{cases}.\]
Hence, if $f_1,\dots, f_r:G\to R$ are functions and $g_1,\dots, g_r$ distinct elements of $G$, then
\[\left(\phi \left( \sum_{i=1}^r f_i g_i \right)\right)(x,y)=\left(\sum_{i=1}^rD_{f_i}A_{g_i}\right )(x,y)=\begin{cases} f_i(x) &\mbox{if}\ y=g_i^{-1}x\\ 0 &\mbox{if}\ y\neq g_i^{-1}x \end{cases}.\]
From this, we ascertain that $\phi$ is bijective. 
\end{proof}

\subsection{Amenable subsets, paradoxical decompositions, and the F\o lner condition}

In this subsection, we investigate amenable subsets of groups, beginning with a few elementary observations. Notice first that a finite subset of a group is amenable with respect to the group if and only if it is nonempty.  
We also invite the reader to verify using Definition~1.2 that a subset $X$ is amenable with respect to a group $G$ if and only if $X$ is amenable with respect to the subgroup of $G$ that is generated by $X$. (This is, though, most easily shown by applying Theorem~2.19 below instead of the definition.) As a consequence, every amenable subgroup of a group $G$ is necessarily amenable with respect to $G$. 

Next we discuss paradoxical decompositions and their relevance to nonamenable subsets, borrowing the parlance and approach adopted by G. Tomkowicz and S. Wagon in their treatise \cite{Wagon} on the Banach-Tarski paradox. First we state the following definition from their book. 

\begin{definition}[{\cite[Definition 3.4]{Wagon}}] {\rm Let $G$ be a group. Two sets $A, B\subseteq G$ are \emph{G-equidecomposable}, written $A\sim_G B$, if there are a partition $\{A_1,\dots,A_n\}$ of $A$ and elements $g_1,\dots, g_n\in G$ such that $\{g_1A_1,\dots, g_nA_n\}$ is a partition of $B$.}
\end{definition}

It is easy to see that $\sim_G$ is an equivalence relation on the set of subsets of $G$. Our next lemma provides an alternative description of $G$-equidecomposability, one well suited for some of the arguments that appear later. 

\begin{lemma} Let $G$ be a group and $A, B$ subsets of $G$. Then $A\sim_G B$ if and only if there are a finite set $K\subseteq G$ and a bijective map $\alpha:A\to B$ such that $\alpha(a)a^{-1}\in K$ for all $a\in A$. 
\end{lemma} 

\begin{proof} Assume $A\sim_G B$. This means that there are a finite set $K\subseteq G$ and a partition $\{A_k:k\in K\}$ of $A$ such that $\{kA_k:k\in K\}$ is a partition of $B$. Then the map $\alpha:A\to B$ defined by $\alpha(a_k)=ka_k$ for every $a_k\in A_k$ has the required properties. 

For the converse, we assume that there are a set $K$ and a bijection $\alpha:A\to B$ as described. For each $k\in K$, set $A_k:=\{a\in A:\alpha(a)=ka\}$. Then $\{A_k:k\in K\}$ is a partition of $A$, and $\{kA_k:k\in K\}$ is a partition of $B$. Therefore $A\sim_G B$.
\end{proof}

The notion of equidecomposability of subsets of groups provides a convenient way to define paradoxical subsets of groups. 

\begin{definition}[{\cite[Definition 1.1]{Wagon}}] {\rm Let $G$ be a group. A subset $X$ of $G$ is said to be \emph{$G$-paradoxical} if there are sets $A, B\subseteq X$ such that $A\cap B=\emptyset$ and $A\sim_G X\sim_G B$.}
\end{definition}

It is shown in \cite{Wagon} that the sets $A$ and $B$ in Definition 2.17 can always be chosen so that their union is $X$. 

\begin{lemma}[{\cite[Corollary 3.7]{Wagon}}] A subset $X$ of a group $G$ is $G$-paradoxical if and only if there is a partition $\{A, B\}$ of $X$ such that $A\sim_G X\sim_G B$.
\end{lemma}

A partition $\{A, B\}$ of a subset $X$ of a group $G$ such that $A\sim_G X\sim_G B$ is called a \emph{paradoxical decomposition} of $X$ with respect to $G$. A seminal result by Tarski asserts that the existence of a paradoxical decomposition of a subset of a group is equivalent to the failure of the subset to be amenable.

\begin{theorem}[{Tarski \cite{Tar}, \cite[Corollary 11.2]{Wagon}}] Let $G$ be a group and $X$ a subset of $G$. Then the following two statements are equivalent.
\begin{enumerate*}
\item $X$ is $G$-paradoxical.
\item $X$ is not amenable with respect to $G$.
\end{enumerate*}
\end{theorem} 

An important characterization of amenable groups was established by F\o lner \cite{Folner} in 1955. Rosenblatt \cite[Theorem 4.9]{Rosenblatt1} extended F\o lner's result in 1973 to describe the amenability of a subset of a group. For our purposes, it will be convenient to use the following version of Rosenblatt's condition. 

\begin{definition}{\rm Let $G$ be a group. If $X$ is a subset of $G$, then we say that {\it $G$ satisfies the F\o lner condition with respect to $X$} if, for any finite subset $K$ of $G$ and $\epsilon\in \mathbb R^+$, there exists a finite subset $F$ of $G$ such that}
\[\lvert KF\cap X\rvert<\left (1+\epsilon \right )\lvert F\cap X\rvert.\]
\end{definition}

In Theorem 2.21, we show that the above condition characterizes amenability of subsets of groups.

\begin{theorem} Let $G$ be a group and $X$ a subset of $G$. Then the following two statements are equivalent.
\begin{enumerate*}
\item $X$ is amenable with respect to $G$.
\item $G$ satisfies the F\o lner condition with respect to $X$. 
\end{enumerate*}
\end{theorem} 

The assertion (i)$\Longrightarrow$(ii) in Theorem ~2.21 follows immediately from \cite[Theorem~4.9]{Rosenblatt1}, which is proved using analysis. Below we supply an alternative justification for this implication that is, instead, graph-theoretic; it is an adaptation of the argument for \cite[Theorem 4.9.2((c)$\Longrightarrow$(e))]{CSC}. 
For this approach, we require the notion of a \emph{bipartite graph}, which is a triple $\mathcal G=(V, W, E)$ where $V$ and $W$ are sets and $E\subseteq V\times W$. The elements of $V$ and $W$ are the \emph{left vertices} and \emph{right vertices}, respectively, and the elements of $E$ are the \emph{edges}. For any $v_0\in V$ and $w_0\in W$, we define the \emph{right neighborhood} $\mathcal{N}_ R(v_0)$ of $v_0 $ and the \emph{left neighborhood} $\mathcal{N}_ L(w_0)$ of 
$w_0$ as follows:
\[\mathcal{N}_ R(v_0):=\{w\in W\ :\ (v_0,w)\in E\};\ \ \ \mathcal{N}_ L(w_0):=\{v\in V\ :\ (v,w_0)\in E\}.\]
For subsets $A\subseteq V$ and $B\subseteq W$, we write
\[\mathcal{N}_ R(A):=\bigcup_{a\in A} \mathcal{N}_ R(a),\ \ \ \mathcal{N}_ L(B):=\bigcup_{b\in B} \mathcal{N}_ L(b).\]
If the left and right neighborhoods of each vertex are finite, then the graph is said to be \emph{locally finite}. 

For a bipartite graph $\mathcal G=(V, W, E)$ and a positive integer $n$, a \emph{perfect one-to-$n$ matching} is a subset $M$ of $E$ satisfying the following two conditions.
\begin{enumerate*}
\item For each $v\in V$, there are exactly $n$ elements $w\in W$ such that $(v, w)\in M$.
\item For each $w\in W$, there is a unique $v\in V$ such that $(v,w)\in M$. 
\end{enumerate*}
Invoking the Axiom of Choice, we ascertain that the existence of a perfect one-to-$n$ matching is equivalent to the existence of a partition $\{A_1,\dots,A_n\}$ of $W$ together with bijections $\alpha_i:V\to A_i$, $i=1,\dots,n$, such that $(v,\alpha_i(v))\in E$ for every $v\in V$ and $i=1,\dots,n$. 

Our proof of Theorem~2.21 relies on the following theorem, which was proved for finite graphs by P. Hall \cite{PHall} in 1935 and generalized to locally finite graphs by M. Hall \cite{MHall} in 1948.  

\begin{theorem}[{P. Hall, M. Hall, \cite[Theorem H.4.2]{CSC}}] Let $n\in \mathbb Z^+$ and let $\mathcal G=(V, W, E)$ be a locally finite bipartite graph. Then $\mathcal G$ admits a perfect one-to-$n$ matching if and only if $\mathcal G$ satisfies the following two conditions.
\begin{enumerate*}
\item For every finite set $A\subseteq V$, $\displaystyle{\lvert \mathcal N_R(A)\rvert \geq n\lvert A\rvert}$.
\item For every finite set $B\subseteq W$, $\displaystyle{\lvert \mathcal N_L(B)\rvert \geq \frac{1}{n}\lvert B\rvert}$.
\end{enumerate*}
\end{theorem}

\begin{proof}[Proof of Theorem 2.21] 
We prove (i)$\Longrightarrow$(ii) by establishing the contrapositive. Suppose, therefore, that $G$ fails to satisfy the F\o lner condition with respect to $X$. This means that there is a finite subset $K$ 
of $G$ and a positive real number $\epsilon$ such that $\lvert KF\cap X\rvert\geq \left (1+\epsilon\right )\lvert F\cap X\rvert$ for every finite set $F\subseteq G$. Hence $\lvert K^iF\cap X\rvert\geq \left (1+\epsilon \right )^i\lvert F\cap X\rvert$ for every $i\in \mathbb Z^+$ and finite set $F\subseteq G$. Therefore, by replacing $K$ with one of its powers, we can assume that $\lvert KF\cap X\rvert\geq 2\lvert F\cap X\rvert$ for every finite subset $F$ of $G$. By adding elements to $K$, we can assume further that $K^{-1}=K$.

Let $\mathcal G$ be the locally finite bipartite graph $(X, X, E)$ where $E:=\{(x,y)\in X\times X: y\in Kx\}$. If $F$ is an arbitrary finite subset of $X$, then
\[\lvert \mathcal N_R(F)\rvert = \lvert \mathcal N_L(F)\rvert = \lvert KF\cap X\rvert \geq 2\lvert F\rvert.\]
Therefore $\mathcal G$ satisfies conditions (i) and (ii) in Theorem 2.22 for $n=2$. As a result, we have a partition $\{A,B\}$ of $X$ and bijections $\alpha:X\to A$, $\beta:X\to B$ such that 
$\alpha(x)x^{-1}, \beta(x)x^{-1}\in K$ for every $x\in X$. It follows, then, from Lemma 2.16 that $A\sim_G X\sim_G B$. In other words, $X$ is $G$-paradoxical. Hence Theorem 2.19 implies that $X$ is not amenable with respect to $G$. 

To show the implication (ii)$\implies$(i), we again prove the contrapositive. Assume, then,
that $X$ is not amenable with respect to $G$. Appealing to Theorem~2.19 and Lemma~2.16, we obtain a finite set $K\subseteq G$, disjoint subsets $A, B$ of $X$, a bijection $\alpha:X\to A$, and a bijection $\beta:X\to B$ such that $\alpha(x)x^{-1}, \beta(x)x^{-1}\in K$ for every $x\in X$. Let $F$ be an arbitrary finite subset of $G$. Since $\alpha(F\cap X)\subseteq KF\cap A$ and $\beta(F\cap X)\subseteq KF\cap B$, we deduce that $\lvert KF\cap X\rvert\geq 2\lvert F\cap X\rvert$. 
This implies that $G$ fails to satisfy the F\o lner condition with respect to $X$. 
\end{proof}

\section{Graded rings with UGN}

In this section, we establish our main results on graded rings, Theorem~A, Theorem~B, Corollary~A, Theorem~C, and Theorem~D. We conclude the section by presenting two applications of Theorem B: in the first (Corollary 3.9), we show that the Weyl rings over a UGN ring have UGN; in the second (Corollary 3.10), we prove that a twisted group ring of a free-by-amenable group over a UGN-ring must possess UGN. 

\subsection{Theorem A, right amenability, Corollary A, and Theorem C}

\begin{proof}[Proof of Theorem A] The ``only if" part follows from Lemma 2.7. We will establish the ``if" statement by proving its contrapositive. Suppose that $R$ has BGN. This means, by Proposition 2.2((i)$\implies$(iii)), that there is an $R$-module epimorphism $\phi:R^n\to R^m$ for some pair of positive integers $m,n$ such that $m>nr$. Put $X:={\rm Supp}(R)$ and
let $K\subseteq X$ be the union of the supports of all the entries in the matrix representation of $\phi$ with respect to the standard bases. By Theorem 2.21, there is a finite subset $F$ of $G$
such that, writing  $U:=K^{-1}F$, we have

\begin{equation} |U\cap X| < \frac{m}{nr}|F\cap X|.\end{equation}

For any positive integer $t$ and subset $V$ of $G$, let $R^t_{V}$ be the $R_1$-submodule of $R^t$ consisting of all the elements with support contained in $V$. In addition, let 
$\pi^t_V~:~R^t\to~R^t_V$ be the projection map and notice that this map is an $R_1$-module epimorphism. We claim that $R_F^m=\pi^m_F\, \phi\left (R^n_{U}\right )$. To show this, let $y\in R^m_F$. Then $y=\phi(x+x')$, where $x\in R^n_{U}$ and $x'\in R^n$ with ${\rm Supp}(x')\cap U=\emptyset$. But ${\rm Supp}[\phi(x')]\subseteq K\cdot {\rm Supp}(x')$. Thus ${\rm Supp}[\phi(x')]\cap F=\emptyset$; that is, $\pi^m_F\, \phi(x')=0$. Therefore $y=\pi_F^m\, \phi(x)$, proving the claim.  

Let $\phi_{U}:R^n_{U}\to R^m$ be the restriction of $\phi$ to $R^n_{U}$, which is an $R_1$-module homomorphism. In view of the claim made in the preceding paragraph, the composition $\pi^m_F\, \phi_{U}: R^n_U\to R^m_F$ is an $R_1$-module epimorphism. Moreover, applying condition (i), we ascertain that there is an $R_1$-module isomorphism $R^n_{U}\to \left (R_1\right )^k$ for some $k\in \mathbb Z^+$ with $k\leq nr |U\cap X|$. Similarly, there is an integer $l\geq m|F\cap X|$ such that $R^m_F\cong \left ( R_1\right )^l$ as $R_1$-modules. Thus $\pi^m_F\, \phi_{U}$ induces an $R_1$-module epimorphism $\left (R_1\right )^k\to \left (R_1\right )^l$. But it follows from (3.1) that $k<l$, which implies that $R_1$ has BGN.
\end{proof}

\begin{remark}{\rm We point out that the hypothesis (i) in Theorem A can be replaced by the more general statement (i$'$) below. 
\vspace{5pt}

\noindent (i$'$) There is a positive integer $s$ such that, for every $g\in {\rm Supp}(R)$, the following two statements hold.
\begin{itemize}
\item $R_g$ can be generated as an $R_1$-module by $s$ elements.
\item The $R_1$-module $\left (R_g\right )^s$ contains a copy of $R_1$ as a direct summand.
\end{itemize}}
\end{remark}
To prove this, one can use essentially the same argument as for Theorem A, but with $m$ chosen so that $m>ns^2$ and $r$ in (3.1) replaced by $s^2$. In this case, the key $R_1$-module epimorphism to consider is one from $R^{ns}_U$ to $R^{ms}_F$ induced by $\pi^m_F\, \phi_U$. 
\vspace{5pt}

Our notion of an amenable subset of a group should more precisely be referred to as \emph{left amenability} since it requires the left invariance of the measure (see Definition 1.2). If this is replaced by right invariance, then the subset can be said to be \emph{right amenable} with respect to the group. It is then obvious that a subset $X$ of a group $G$ is right amenable with respect to $G$  if and only if $X^{-1}$ is left amenable with respect to $G$. This observation allows us to enunciate a dual version of Theorem A.

\begin{corollary} Let $G$ be a group and $R$ a ring graded by $G$ such that the following two conditions hold.
\begin{enumerate*} 
\item There is a positive integer $r$ such that, for every $g\in G$,  $R_g\cong (R_1)^i$ as left $R_1$-modules for some $i=0,\dots, r$.
\item The support of $R$ is right amenable with respect to $G$.
\end{enumerate*}
Then $R$ has UGN if and only if $R_1$ has UGN. 
\end{corollary}

\begin{proof} The ``only if" statement follows from Lemma 2.7. To prove the converse, we employ the ring $R^{\rm op}$ and grade it by $G$ in the following way: $\left (R^{\rm op}\right )_g:=R_{g^{-1}}$ for every $g\in G$. Note that we then have $\left (R^{\rm op}\right )_1=\left (R_1\right )^{\rm op}$ as rings. Moreover, for each $g\in G$, there is an integer $i\in [0,r]$ such that $\left (R^{\rm op}\right )_g\cong \left (\left (R^{\rm op}\right )_1\right )^i$ as right $\left (R^{\rm op}\right )_1$-modules. 

Suppose that $R_1$ has UGN. Then $\left (R^{\rm op}\right )_1$ has UGN by Lemma 2.6. Also, since ${\rm Supp}\left (R^{\rm op}\right )=\left [{\rm Supp}\left (R\right )\right ]^{-1}$, we have that ${\rm Supp}\left (R^{\rm op}\right )$ is a left amenable subset of $G$.  Therefore, by Theorem A, $R^{\rm op}$ has UGN.  It follows, then, from Lemma 2.6 that $R$ has UGN.
\end{proof}

With the aid of Corollary 3.1, we can prove Corollary A very easily.

\begin{proof}[Proof of Corollary A] The case where the homogenous components are free right $R_1$-modules of bounded rank follows immediately from Theorem A; the left-module case is a consequence of Corollary 3.1.
\end{proof}

The best known examples of nonempty subsets of amenable groups that are neither left amenable nor right amenable are free subsemigroups on two generators. An elementary amenable group contains such a subsemigroup if and only if it is not locally virtually nilpotent (see \cite[Theorem 3.2$'$]{Chou}), that is, not supramenable.  Below we present an example \cite{Lorensen}, suggested to the first author by Nicolas Monod, of a subset of an amenable group that is right amenable but not left amenable. 

\begin{example} {\rm Let $k>1$ be an integer and $G$ the Baumslag-Solitar group $\mathrm{BS}(1,k)$; that is, 
\[G:=\langle a, b\ :\ bab^{-1}=a^k\rangle.\]
Let $A$ and $B$ be the infinite cyclic subgroups generated by $a$ and $b$, respectively. Notice that $G=A^G\rtimes B$, where $A^G\cong \mathbb Z[1/k]$ is the normal closure of $A$ in $G$. 
Since $G$ is metabelian, it is an amenable group. 

Put $X:=AB$. We claim that $X$ is right amenable but not left amenable.
To prove the latter assertion, write $X_0:=\langle a^k\rangle B$. Observe that $X_0$ and $aX_0$ are disjoint subsets of $X$, and that $bX=X_0$. Hence the existence of a function $\mu:\mathcal{P}(G)\to [0,\infty]$ satisfying properties (i), (ii), and (iii) in Definition 1.2 would imply
$$2=\mu(X)+\mu(X)=\mu(X_0)+\mu(X_0)=\mu(X_0)+\mu(aX_0)\leq \mu(X)=1.$$ 
Therefore $X$ is not left amenable with respect to $G$. 

To establish the right amenability of $X$, we invoke the following proposition, due to Rosenblatt \cite{Rosenblatt1}. He states the proposition in terms of left amenability; we have repurposed it for right amenability. Following Rosenblatt, we employ the following notation: for any $n$-tuple $(u_1,\dots,u_n)$ and any set $X$, 
$$\|X\cap(u_1,\dots,u_n)\|:=|\{ i=1,\dots, n:u_i\in X\}|.$$

\begin{proposition}[{Rosenblatt \cite[Corollary 1.5]{Rosenblatt1}}] Let $G$ be an amenable group. A subset $X$ of $G$ is right amenable with respect to $G$ if and only if, for every $m$-tuple $(u_1,\dots, u_m)$ and $n$-tuple $(v_1,\dots,v_n)$ from $G$ such that $m<n$, there exists $g\in G$ such that 
\[ \|gX\cap (u_1,\dots,u_m)\|<\|gX\cap (v_1,\dots, v_n)\|.\]
\end{proposition}

To complete the argument using this proposition, let $\{t_{\alpha}:\alpha\in I\}$ be a complete set of coset representatives of $A$ in $A^G$. Then 
$\{t_{\alpha}X:\alpha \in I\}$ is a partition of $G$. In order to apply Rosenblatt's proposition, we let $(u_1,\dots,u_m)$ and $(v_1,\dots, v_n)$ be finite sequences from $G$ with $m<n$. Then 
$$m=\sum_{\alpha\in I} \|t_{\alpha}X\cap (u_1,\dots, u_m)\|\ \ \mbox{and}\ \  n=\sum_{\alpha\in I} \|t_{\alpha}X\cap (v_1,\dots, v_n)\|.$$
Thus it must be the case that, for some $\alpha\in I$, 
$$\|t_{\alpha}X\cap (u_1,\dots,u_m)\|<\|t_{\alpha}X\cap (v_1,\dots, v_n)\|.$$
Hence $X$ is right amenable with respect to $G$.}
\end{example}

Next we establish Theorem C, dealing with projective gradings by locally finite groups. For the concepts from Morita theory invoked in the proof, we refer the reader to either \cite[\S 6]{Anderson-Fuller} or \cite[\S 18]{Lam}. 

\begin{proof}[Proof of Theorem C] The ``only if" portion of the statement follows from Lemma 2.7. To prove the converse, we assume that $R_1$ has UGN. In view of Lemma 2.6, we only need to consider the case where the homogeneous components of $R$ are finitely generated projective right $R_1$-modules. Also, because UGN is inherited by direct limits (see Lemma 2.10), it suffices to treat the case where $G$ is finite. This means that $R$ is finitely generated and projective as an $R_1$-module. Since $R_1$ is an $R_1$-module direct summand in $R$, we have further that $R$ is a generator for $\mathfrak M_{R_1}$.  These properties make $R$ a progenerator in $\mathfrak M_{R_1}$, rendering ${\rm End}_{R_1}(R)$ Morita equivalent to $R_1$. It follows, then, from Corollary 2.3 that ${\rm End}_{R_1}(R)$ has UGN. Therefore, being isomorphic to a subring of this endomorphism ring, $R$ must also have UGN.
\end{proof} 

\subsection{Theorem D and its corollaries}

Before proving Theorem D, we establish the following property of translation rings, generalizing both the result and argument contained in \cite[\S 3]{Elek} (see also \cite[Proposition~2.2]{Elek2}).

\begin{proposition} Let $G$ be a group and $X$ a subset of $G$. If $X$ is not amenable with respect to $G$, then, for any ring $R$, \[\left (T_G(X,R)\right )^2\cong  T_G(X,R)\] as $T_G(X,R)$-modules, so that ${\rm gn}\left (T_G(X,R)\right )=1$.
\end{proposition}

\begin{proof}
 Assume that $X$ is not amenable with respect to $G$.  If $X=\emptyset$, then the conclusion is clearly true. Suppose $X\neq \emptyset$. Applying Theorem~2.19, Lemma~2.18, and Lemma~2.16, we acquire a finite set $K\subseteq G$, a partition $\{A,B\}$ of $X$, and bijections $\alpha:X\to A$ and $\beta:X\to B$ such that $\alpha(x)\, x^{-1},\ \beta(x)\, x^{-1}\in K$ for all $x\in X$.  
 Now let $R$ be a ring and define the elements $M$ and $N$ of $T:=T_G(X,R)$ as follows: 
\[M(x,y):=\begin{cases} 1 & \mbox{if}\ y=\alpha(x)\\ 0 & \mbox{if}\ y\neq \alpha(x) \end{cases},\ \ \ \ \ N(x,y):=\begin{cases} 1 & \mbox{if}\ y=\beta(x)\\ 0 & \mbox{if}\ y\neq \beta(x)\end{cases}\]
for all $x, y\in X$. 

Notice that the ring $T_G(X,R)$ is closed under transposition, so that $M^t, N^t\in T_G(X,R)$.
 Our aim is to establish the equations (3.2) and (3.3) below; these will imply $T^2\cong T$ as $T$-modules. 
\begin{equation}\begin{pmatrix}M\\
N \end{pmatrix} \begin{pmatrix}M^t & N^t\end{pmatrix}=\begin{pmatrix} 1_T & 0\\
0 & 1_T\end{pmatrix}.\end{equation}
\begin{equation} \begin{pmatrix}M^t & N^t
 \end{pmatrix} \begin{pmatrix}M\\
 N\end{pmatrix}=\begin{pmatrix} 1_T\end{pmatrix}. \end{equation}

Let $x, x'$ be arbitrary elements of $X$. Then
\[(MM^t)(x, x')=\sum_{y\in X}M(x, y)M(x',y).\]
If $x=x'$, then this sum is plainly equal to $1$. If $x\neq x'$, then $\alpha(x)\neq \alpha(x')$, which means $(MM^t)(x,x')=0$. Hence $MM^t=1_T$. Similarly, $NN^t=1_T$.  
Moreover, 
\[(MN^t)(x,x')=\sum_{y\in X} M(x,y)N(x',y).\]
Since $\alpha(x)\neq \beta(x')$, this sum is equal to $0$. Thus $MN^t=0$. By the same token, $NM^t=0$. We have shown, then, that (3.2) holds.

Now we calculate the products $M^tM$ and $N^tN$. 
\[(M^tM)(x,x')=\sum_{y\in X}M(y,x)M(y,x')\]
for all $x, x'\in X$. This sum is equal to $1$ if there is a $y\in X$ such that $x=x'=\alpha(y)$; otherwise it is $0$. Hence $M^tM$ is diagonal with
\[(M^tM)(x,x)=\begin{cases} 1 & \mbox{if}\ x\in  A\\ 0 & \mbox{if}\ x\in B \end{cases}\]
for every $x\in X$. Similarly, $N^tN$ is diagonal, and 
\[(N^tN)(x,x)=\begin{cases} 0 & \mbox{if}\ x\in A\\ 1 & \mbox{if}\ x\in B \end{cases}\]
for all $x\in X$. Equation (3.3), then, follows. 
\end{proof}

\begin{proof}[Proof of Theorem D] The implication (ii)$\Longrightarrow$(i) follows from Proposition 3.4. For the proof of the converse, we denote the set $\{1,\dots,k\}$ by $J_k$ for any $k\in \mathbb Z^+$. Assume that (i) holds, and let $R$ be a ring. Notice that the ``only if" part of (ii) is a consequence of Lemma~2.7. To prove the ``if" part, we suppose that $T_G(X,R)$ possesses BGN and deduce that $R$ must also have BGN. Invoking Proposition~2.2((i)$\Longrightarrow$(ii)) and Proposition~2.5, we obtain 
two integers $m>n>0$, a $J_m\times J_n$ matrix $A$ with entries in $T_G(X,R)$, and a $J_n\times J_m$ matrix $B$ with entries in $T_G(X,R)$ such that $AB$ is the $J_m\times J_m$ identity matrix over $T_G(X,R)$. For each pair $i\in J_m, j\in J_n$, there are finite subsets $K_{ij}$ and $L_{ji}$ of $G$ such that $A_{ij}(x,y)=0$ whenever $y\notin K_{ij}x$ and $B_{ji}(x,y)=0$ whenever $y\notin L_{ji}x$. Now put 

\[K:=\{ 1\}\cup \bigcup_{i,j} K_{ij} \cup \bigcup_{i,j} K_{ij}^{-1} \cup \bigcup_{i,j} L_{ji} \cup \bigcup_{i,j} L_{ji}^{-1} .\]
Observe that, for all $i\in J_m$ and $j\in J_n$, we have 
 
 \[A_{ij}(x,y)=A_{ij}(y,x)=B_{ji}(x,y)=B_{ji}(y,x)=0\ \ \mbox{whenever}\ \ y\notin Kx.\] 
 
 By Theorem 2.21, $G$ satisfies the F\o lner condition with respect to $X$. This means that there is a finite subset $F$ of $G$ such that $n\lvert KF\cap X\rvert<m\lvert F\cap X\rvert$. Write $U:=KF \cap X$ and $F_X:=F\cap X$.  
 Let $A^\ast$ be the $J_m\times F_X$ by $J_n\times U$ matrix with entries given by

\[A^\ast((i,f),(j,u)):=A_{ij}(f,u)\]
for all $i\in J_m, f\in F_X, j\in J_n, u\in U$.
Moreover, let $B^\ast$ be the $J_n\times U$ by $J_m\times F_X$ matrix defined by

\[B^\ast((j,u),(i,f)):= B_{ji}(u,f)\]
for all $j\in J_n, u\in U, i\in J_m, f\in F_X$.

With the above definitions, we now look at the product $A^\ast B^\ast$ and show that it is the $J_m\times F_X$ by $J_m\times F_X$ identity matrix. Let $i, i'\in J_m$ and $f, f'\in F_X$. Then 

\[\left (A^\ast B^\ast\right )\left (\left (i,f\right ), \left (i', f'\right )\right )
=\sum_{\substack{u\in\ Kf\cap Kf'\cap X \\ j = 1,\dots, n}} A_{ij}\left (f,u\right )B_{ji'}\left (u,f'\right )\]

\noindent if $Kf\cap Kf'\cap X\neq \emptyset$; otherwise $\left (A^\ast B^\ast\right )\left (\left (i,f\right ), \left (i', f'\right )\right )=0$. Hence, in the former case, we have

\begin{equation*}
\begin{split}
\left (A^\ast B^\ast\right )\left (\left (i,f\right ), \left (i', f'\right )\right ) & = \sum_{j=1}^n \left (A_{ij}B_{ji'}\right )(f,f') \\ 
& = \left (\sum_{j=1}^n A_{ij}B_{ji'}\right )(f,f') \\ 
& = \left (AB\right )_{ii'}(f,f') \\
& = \begin{cases} 0, & i\neq i'\\
0, & i=i', f\neq f'\\
1, & i=i', f=f'.\end{cases}
\end{split}
\end{equation*}
Moreover, if $Kf\cap Kf'\cap X= \emptyset$, then $f\neq f'$.  It follows, therefore, that  
$A^\ast B^\ast$ is indeed the $J_m\times F_X$ by $J_m\times F_X$ identity matrix. Also, $A^\ast$ has $m\lvert F_X\rvert$ rows and $n\lvert U\rvert$ columns, and $B^\ast$ has $n\lvert U\rvert$ rows and $m\lvert F_X\rvert$ columns. Hence, since  $n\lvert U\rvert<m\lvert F_X\rvert$, the ring $R$ has BGN by virtue of Lemma 2.5 and Proposition 2.2((ii)$\Longrightarrow$(i)).  
\end{proof}

Theorem D supplies right away the following characterizations of amenability and supramenability.

\begin{corollary} For any group $G$, the following two assertions are equivalent.
\begin{enumerate*}
\item $G$ is amenable.
\item For any ring $R$, the ring $T(G,R)$ has UGN if and only if $R$ has UGN.
\end{enumerate*}
\end{corollary}

\begin{corollary} For any group $G$, the following two assertions are equivalent.
\begin{enumerate*}
\item $G$ is supramenable.
\item For any ring $R$ and nonempty set $X\subseteq G$, the ring $T_G(X,R)$ has UGN if and only if $R$ has UGN.
\end{enumerate*}
\end{corollary}

The translation ring $T_G(X,R)$ that we have defined should properly be called  a \emph{left translation ring} since the dual concept of a \emph{right translation ring}, denoted $T_G^r(X,A)$, can be defined as the set of all $X\times X$ matrices $M$ for which there is a finite set $K\subseteq G$ such that $M(x,y)=0$ whenever $y\notin xK$ . Notice that the map $M\mapsto M^\ast$ from $T^r_G(X,R)$ to $T_G(X^{-1},R)$, where $M^\ast(x^{-1},y^{-1}):=M(x,y)$ for all $x,y\in X$, is an isomorphism of rings. 
This observation leads immediately to the following dual version of Theorem D. 

\begin{corollary} Let $G$ be a group and $X$ a subset of $G$. Then the following two statements are equivalent.
\begin{enumerate*}
\item The subset $X$ is right amenable with respect to $G$.
\item  For every ring $R$, the ring $T^r_G(X,R)$ has UGN if and only if $R$ has UGN.
\end{enumerate*}
\end{corollary}

\subsection{Theorem B and its corollaries}

We conclude this section by proving Theorem B and describing two applications. 

\begin{proof}[Proof of Theorem B] First we establish (i)$\Longrightarrow$(ii). Assume that $G$ is amenable, and let $R$ be a ring with a $G$-grading that is full and boundedly free. By Lemma 2.7, $R_1$ has UGN if $R$ has UGN. To establish the converse, suppose that $R_1$ has UGN. If the homogeneous components are free right $R_1$-modules of bounded rank, then it follows immediately from Theorem A that $R$ has UGN. For the left-module case, we just need to apply Corollary 3.1.
 
The assertions (ii)$\Longrightarrow$(iii) and (iii)$\Longrightarrow$(iv) are trivial. 
For (iv)$\Longrightarrow$(i), we prove the contrapositive. Suppose that $G$ is not amenable and let $S$ be a ring with UGN. According to Proposition 3.4, the ring $T(G,S)$ has BGN. Furthermore, Proposition 2.14 implies that $T(G,S)$ is a skew group ring $\left (\prod_G S\right )~\ast ~G$. Also, by Lemma 2.11, the ring  $\prod_G S$ has UGN. Therefore statement (iv) fails to hold.
\end{proof}
 
The best known examples of rings to which Theorem~B can be applied are crossed products. But the value of the theorem extends well beyond this realm, for there are many other sorts of rings that permit full and boundedly free gradings by amenable groups.  One well-known family of such rings are the Weyl rings.

\begin{definition}{\rm Let $R$ be a ring. The \emph{first Weyl ring} over $R$, denoted $W_1(R)$, is the $R$-ring with the presentation

\[W_1(R):=\langle x,y\ :\ xy - yx =1;\ \ rx=xr\ \mbox{and}\ ry=yr\ \mbox{for all}\ r\in R\rangle.\]

For any integer $n>1$, the \emph{$n$th Weyl ring} $W_n(R)$ over $R$ is defined by $W_n(R):=W_1\left (W_{n-1}\left (R\right )\right )$.}
\end{definition}

\begin{corollary} If $R$ is a ring with UGN, then $W_n(R)$ has UGN for every $n\in \mathbb Z^+$.
\end{corollary}

\begin{proof} It suffices to show that $S:=W_1(R)$ has UGN. For this, we grade $S$ by $\mathbb Z$ by assigning ${\rm deg}(x):=1$ and ${\rm deg}(y):=-1$. 
An appeal to \cite[Corollary 6.2]{diamond} allows us to deduce that $S$ is freely generated as an $R$-module by both the set of monomials $x^k\, y^l$ for $k, l=0, 1, 2, \dots$ and the set of monomials $y^l\, x^k$ for $k, l=0, 1, 2, \dots$ . As a result, $S_0$ is a free $R$-module on the sets $\mathcal{B}:=\{x^k\, y^k\ :\ k = 0,1,2,\dots\}$ and $\mathcal{B'}:=\{y^k\, x^k\ :\ k = 0,1,2,\dots\}$. From these observations, we can see that, if $n\in \mathbb Z^+$, then $S_n$ and $S_{-n}$ are freely generated as $R_0$-modules by $x^n$ and $y^n$, respectively. 

Whenever the product of any two elements of $\mathcal{B}-\{1\}$ is expressed as an $R$-linear combination of elements of $\mathcal{B}$, the coefficient of $1$ is $0$. Thus we can define a ring homomorphism $S_0\to R$ by mapping every element $s\in S_0$ to the coefficient of $1$ in the expression of $s$  as a linear combination of elements of $\mathcal{B}$. Lemma 2.7 implies, then, that $S_0$ has UGN. Therefore, by Theorem B, $S$ has UGN.
\end{proof}

It is worth mentioning that there is no counterpart to the characterizations (iii) and (iv) in Theorem B pertaining to twisted group rings. We show this in Corollary~3.10 below by proving that any twisted group ring of a free-by-amenable group over a UGN-ring must have UGN. It is an open question whether the corresponding statement is true for the twisted group ring of an arbitrary group. 

\begin{corollary} Let $G$ be a free-by-amenable group, R a ring, and $R\ast G$ a twisted group ring. Then $R\ast G$ has UGN if and only if $R$ has UGN. 
\end{corollary}

\begin{proof} The ``only if" part follows from Lemma 2.7. To prove the ``if" portion, suppose that $R$ has UGN. The group $G$ has a free normal subgroup $F$ such that $Q:=G/F$ is amenable. This means that the twisted group ring $R\ast G$ is isomorphic to a crossed product of the group $Q$ over a twisted group ring $R\ast F$.  Since $F$ is free, we have $H^2(F,Z(R)^\ast)=0$, which implies $R\ast F\cong RF$ as rings. Hence, by Corollary 2.8, $R\ast F$ possesses UGN. It follows, then, from Theorem~B that $R\ast G$ has UGN. 
\end{proof}

\begin{openquestion2} Let $G$ be an arbitrary group. If $R$ is a ring with UGN, must every twisted group ring $R\ast G$ have UGN?
\end{openquestion2}

\section{Examples of graded rings with finite generating numbers}

In this final section, we discuss obstacles to generalizing Theorem~B, Corollary~A, and Theorem~C in various ways. In \S 4.1, we describe three examples of BGN-rings graded by infinite groups whose base rings have UGN, illustrating that the boundedness and freeness conditions on the grading in Theorem~B((i)$\implies$(ii)) and Corollary~A are necessary. Moreover, in \S 4.2, we present three examples of rings graded by finite groups; the first two show that the equalities between infinite generating numbers that we established in \S 3 fail to translate into equalities between finite generating numbers. The purpose of the last example is to demonstrate that the projectivity condition in Theorem C cannot be jettisoned.   

\subsection{Gradings by infinite groups}

First we prove that the hypothesis that the grading is boundedly free in Theorem~B((i)$\implies$(ii)) and Corollary~A cannot be weakened to the grading being merely free. 

\begin{theorem} There is a ring $R$ with a full and unboundedly free $\mathbb Z$-grading such that $R_0$ has UGN and $R$ has BGN.
\end{theorem}

\begin{proof} 
Let $S$ be an arbitrary UGN-ring, and let $R$ be the $S$-ring generated by $x_1, x_2, x_3, y_1, y_2, y_3$ subject only to the following relations:

\begin{enumeratenum}
\item $s\, x_i=x_i\, s$ for all $i=1,2,3$ and $s\in S$;
\item $s\, y_i=y_i\, s$ for all $i=1,2,3$ and $s\in S$;
\item $x_i\, y_i + y_i\, x_i = 1$ for $i=1,2,3$;
\item $x_i\, y_j + y_i\, x_j = 0$ whenever $i\neq j$. 
\end{enumeratenum}

Relations (3) and (4) yield the matrix equation
\[\begin{pmatrix}
x_1 & y_1\\
x_2 & y_2\\
x_3 & y_3 \end{pmatrix} 
\begin{pmatrix}
y_1 & y_2& y_3\\
x_1 & x_2 & x_3\end{pmatrix}=\begin{pmatrix} 1 & 0 & 0\\
0 & 1 & 0\\
0 & 0 & 1\end{pmatrix}.\]
\noindent Thus, by Lemma 2.5 and Proposition 2.2((ii)$\implies$(i)), $R$ has BGN. 

Set up a $\mathbb Z$-grading on $R$ by assigning ${\rm deg}(x_i):=1$ and ${\rm deg}(y_i):=-1$. 
As in the proof of Corollary 3.9, we can apply \cite[Corollary 6.2]{diamond} to conclude that $R$ is freely generated as an $S$-module by each of the two sets $\mathcal A$ and $\mathcal A'$ defined below.  
\[\mathcal A: = \Big\{x_{i_1}\dots x_{i_k}\, y_{j_1}\dots y_{j_l}\ :\  k,l\geq 0, \ \ 1\leq i_1,\dots,i_k, j_1,\dots,j_l\leq 3\Big\};\]
\[\mathcal A': = \Big\{y_{j_1}\dots y_{j_l}\, x_{i_1}\dots x_{i_k}\ :\  k,l\geq 0, \ \ 1\leq i_1,\dots,i_k, j_1,\dots,j_l\leq 3\Big\}.\]
As a result, $R_0$ is freely generated as an $S$-module by each of the sets
\[\mathcal{B}: = \Big\{x_{i_1}\dots x_{i_k}\, y_{j_1}\dots y_{j_k}\ :\  k\geq 0, \ \ 1\leq i_1,\dots,i_k, j_1,\dots,j_k\leq 3\Big\};\]
\[\mathcal{B}': = \Big\{y_{j_1}\dots y_{j_k}\, x_{i_1}\dots x_{i_k}\ :\  k\geq 0, \ \ 1\leq i_1,\dots,i_k, j_1,\dots,j_k\leq 3\Big\}.\]

From the properties above, it is apparent that, for $n\in \mathbb Z^+$, $R_n$ is freely generated as an $R_0$-module by the set 
\[\Big\{x_{i_1}\dots x_{i_n}\ :\ 1\leq i_1,\dots,i_n\leq 3\Big\},\]
and that $R_{-n}$ is freely generated as an $R_0$-module by the set
\[\Big\{y_{j_1}\dots y_{j_n}\ :\ 1\leq j_1,\dots,j_n\leq 3\Big\}.\]
Hence the grading that we have defined on $R$ is full and unboundedly free. 

Define the function $\phi:R_0\to S$ by, for any $r\in R_0$, letting $\phi(r)$ be the coefficient of~$1$ in the expression of $r$ as an $S$-linear combination of the elements of $\mathcal{B}$.  Plainly, $\phi$ is additive, and $\phi(1_{R_0})=1_S$. Notice further that, whenever any product of two elements of $\mathcal{B}-\{1\}$ is written as a linear combination of elements of $\mathcal{B}$ with coefficients in $S$ by repeatedly applying relations (3) and (4), the coefficient of $1$ in that linear combination will be $0$. From this observation we can see that $\phi$ is multiplicative and therefore a ring homomorphism. Thus, according to Lemma~2.7, $R_0$ must have UGN.
\end{proof}

\begin{remark}{\rm The authors do not know whether there is an example witnessing Theorem~4.1 that has generating number one.}
\end{remark}

In Theorem 4.2, we demonstrate that the bounded freeness of the grading in Theorem~B((i)$\implies$(ii)) and Corollary~A cannot be replaced by strongness. Since strong gradings are necessarily projective, this result also serves to
show that Theorem~D fails to generalize to supramenable groups.

\begin{theorem} For any positive integer $n$, there exists a strongly $\mathbb Z$-graded ring $R$ such that $R_0$ has UGN and ${\rm gn}(R)=n$. 
\end{theorem}

The main step in establishing Theorem 4.2 is to prove the following proposition.

\begin{proposition} Let $S$ be a UGN-ring, and let $n$ be an integer with $n>1$. Furthermore, let $L$ be the $S$-ring generated by the sets $E:=\{e_1,\dots, e_n\}$ and $E^\ast:=\{e_1^\ast, \dots, e_n^\ast\}$ subject only to the following relations.
\begin{enumerate*}
\item The elements of $S$ commute with the generators in $E\cup E^\ast$. 
\item  $e^\ast_ie_j=\delta_{ij}$ for $i, j=1,\dots, n$.
\item $\sum_{i=1}^n e_ie^\ast_i=1.$
\end{enumerate*}

\noindent Then the following two statements hold.
\begin{itemize}
\item $L^n\cong L$ as $L$-modules, so that ${\rm gn}(L)=1$. 
\item There is a strong $\mathbb Z$-grading on $L$ such that $L_0$ has UGN.
\end{itemize}
\end{proposition}

The $S$-ring $L$ in Proposition 4.3 is often denoted $L_S(1,n)$ (see the remark following the proof of  Proposition 2.12). It is a special case of what is known as a {\it Leavitt path ring} (see \cite{LPA}). Moreover, the $\mathbb  Z$-grading referred to in Proposition 4.3 corresponds to the conventional $\mathbb Z$-grading on such rings (see \cite[\S 2.1]{LPA}). 

The proof of Proposition 4.3 makes use of the following well-known, elementary property. 

\begin{proposition} Let $S$ be a ring, and let $R$ be an $S$-ring that is generated by elements $\epsilon_{ij}$ for $i,j=1,\dots, n$ satisfying the following four conditions:
\begin{enumerate*}
\item $s\epsilon_{ij}=\epsilon_{ij}s$ for all $s\in S$ and $i,j=1,\dots, n$;
\item${\rm Ann}_S(\epsilon_{ij})=0$ for all $i, j=1,\dots, n$;
\item $\epsilon_{ij}\epsilon_{km}=\delta_{jk}\, \epsilon_{im}$ for all $i,j,k,m=1,\dots, n$;
\item $\displaystyle{\sum_{i=1}^n \epsilon_{ii}=1}$.
\end{enumerate*}
If $\{E_{ij}: i,j=1,\dots, n\}$ is the set of standard matrix units in $M_{n}(S)$, then
there is an $S$-ring isomorphism $\phi: M_{n}(S)\to R$ such that
$\phi(E_{ij})=\epsilon_{ij}$ for $i,j=1,\dots, n.$ 
\end{proposition}

\begin{proof}[Proof of Proposition 4.3] Let $\mathcal{M}(E)$ be the multiplicative submonoid of $L$ generated by $E$. Then $\mathcal{M}(E)$ is free on $E$. For any $\alpha=e_{i_1}\cdots e_{i_l}$, we call $l$ the {\it length} of $\alpha$, denoted $l(\alpha)$, and write $\alpha^\ast:=e^\ast_{i_l}e^\ast_{i_{l-1}}\cdots e^\ast_{i_1}$. In addition, we define $l(1):=0$ and $1^\ast:=1$. Finally, we write $\mathcal{M}_l(E)$ for the subset of $\mathcal{M}(E)$ consisting of all the elements of length $l$.  With this notation, we can generalize the relations (ii) and (iii) among the hypotheses of Proposition 4.3 as follows: 
\begin{enumeratenum}
\item  $\alpha^\ast\beta=\delta(\alpha, \beta)$ for all $\alpha, \beta\in \mathcal{M}(E)$ with $l(\alpha)=l(\beta)$;
\item $\displaystyle{\sum_{\alpha \in \mathcal{M}_l(E)}\alpha \alpha^\ast=1}$ for all $l\geq 0$. 
\end{enumeratenum}

We set up the $\mathbb Z$-grading on $L$ by, for any $k\in \mathbb Z$, defining $L_k$ to be the $S$-linear span of the set of elements of the form $\alpha\beta^\ast$ such that $\alpha, \beta\in \mathcal{M}(E)$ and $l(\alpha)-l(\beta)=k$. Moreover, relations (1) and (2) ensure that this grading is strong (see \cite[Proposition~1.1.1(3)]{NO}). 

 That $L^n\cong L$ as $L$-modules follows from the defining relations of $L$. To establish that $L_0$ has UGN, we let $L_0^l$ be the subring of $L_0$ consisting of all $S$-linear combinations of elements of the form $\alpha\beta^\ast$ where $\alpha, \beta\in \mathcal{M}_l(E)$. Then $L_0$ is the union of the chain
\[ L_0^0\subset L^1_0\subset L^2_0\subset \cdots,\]
with the containments following easily from relation (iii). 
Our intent is to apply Proposition 4.4 to show that, for every $l\geq 0$, $L^l_0\cong M_{n^l}(S)$. This will mean, by either Corollary~2.3 or Corollary~3.5, that each subring $L^l_0$ has UGN, implying, by Lemma~2.10, that $L_0$ does too. 

To construct the desired isomorphism, we first choose a bijection $\sigma:\{1, \dots, n^l\}\to \mathcal{M}_l(E)$ and write $\epsilon_{ij}:= \sigma(i) (\sigma(j))^\ast$ for $i,j=1,\dots, n^l$.
Then (1) and (2), respectively, imply the following two equations:
\begin{itemize} \item $\displaystyle{\epsilon_{ij}\epsilon_{km}=\delta_{jk}\, \epsilon_{im}\ \mbox{for}\ \ i,j,k,m=1,\dots, n^l;}$
\item  $\displaystyle{\sum_{i=1}^{n^l} \epsilon_{ii}=1.}$ \end{itemize}
Now let $\{E_{ij}: i,j=1,\dots, n^l\}$ be the set of standard matrix units in $M_{n^l}(S)$. Referring to Proposition 4.4, we see that there is an $S$-ring isomorphism $\phi: M_{n^l}(S)\to L^l_0$ such that
$\phi(E_{ij})= \epsilon_{ij}$ for $i,j=1,\dots, n^l$. 
\end{proof}

\begin{remark} {\rm The interpretation of $L_0$ as a direct limit of matrix rings that plays a central role in our proof of Proposition 4.3 is originally due to G. Abrams and P. N. \'Anh \cite{AbramsAnh} for the case where $S$ is a field.}
\end{remark}

\begin{remark}{\rm The case of Proposition 4.3 when $S$ is a field also follows from an argument constructed by N.Q. Loc \cite[p. 67]{Loc}.}
\end{remark}

\begin{remark}{\rm In light of our interest in free gradings, we observe that, in Proposition~4.3, $L_k$ is a finitely generated free {\it right} $L_0$-module if $k>0$ and a finitely generated free {\it left} $L_0$-module if $k<0$.}
\end{remark}

In order to prove Theorem 4.2, we merely need to form a direct product involving the ring from Proposition~4.3, a technique, borrowed from \cite{Abrams}, that we will use on three further occasions in this section. 

\begin{proof}[Proof of Theorem 4.2] Let $S$ be a ring with UGN, and let $L$ be the $\mathbb Z$-graded ring defined in Proposition~4.3. By Proposition~2.12, there is a ring $U$ with ${\rm gn}(U)=n$. Put $R:=L\times U$ and grade
$R$ by $\mathbb Z$ by defining $R_k:=L_k\times U$ for every $k\in \mathbb Z$. Since the grading on $L$ is strong, the same is true for the grading on $R$. Moreover, Lemma~2.11 implies that $R_0$ has UGN and ${\rm gn}(R)=n$. 
\end{proof}

\begin{remark} {\rm We point out that it is not possible to find a Leavitt path algebra over a field that enjoys the properties of the ring $R$ in Theorem~4.2. The reason is that, if a Leavitt path algebra over a field has BGN, then it must have generating number one (see \cite[Remark 3.17]{AbramsNP}).}
\end{remark}

Theorem 4.2 allows us to establish the following characterization of local finiteness within the class of elementary amenable groups. 

\begin{corollary} For an elementary amenable group $G$, the following two statements are equivalent.
\begin{enumerate*}
\item $G$ is locally finite.
\item Every projectively $G$-graded ring has UGN if and only if its base ring has UGN.
\end{enumerate*}
\end{corollary} 

\begin{proof}The implication (i)$\Longrightarrow$(ii) is Theorem C. To show (ii)$\Longrightarrow$(i), we establish the contrapositive. Hence we suppose that $G$ is not locally finite. It follows, then, from \cite[Theorem 2.3]{Chou} that $G$ contains an infinite cyclic subgroup $H$. According to Theorem~4.2, there is a projectively $H$-graded ring $R$ such that $R_1$ has UGN but $R$ has BGN.  We now extend the $H$-grading on $R$ to a $G$-grading by setting $R_g:=0$ if $g \in G-H$. 
Equipped with this $G$-grading, the ring $R$ witnesses the failure of (ii) to hold. 
\end{proof}

Below we show that the boundedly-free condition in Theorem~B((i)$\implies$(ii)) cannot be replaced by the hypothesis that the grading is merely boundedly projective. 

\begin{theorem} Let $G$ be an elementary amenable group that is not locally virtually nilpotent. For every $n\in \mathbb Z^+$, there exists a ring $R$ graded by $G$ with the following properties.
\begin{enumerate*}
\item The $G$-grading on $R$ is full. 
\item For each $g\in G$, the $R_1$-module $R_g$ is a direct summand in $R_1$. 
\item $R_1$ has UGN.
\item ${\rm gn}(R)=n$.
\end{enumerate*}
\end{theorem}

The main step in establishing Theorem 4.6 is to prove Proposition 4.7.

\begin{proposition}  Let $G$ be any elementary amenable group that is not locally virtually nilpotent. For every $n\in \mathbb Z^+$, there exists a ring $R$ graded by $G$ with the following properties.
\begin{enumerate*}
\item For each $g\in G$, the $R_1$-module $R_g$ is a direct summand in $R_1$. 
\item The base ring $R_1$ has UGN.
\item ${\rm gn}(R)=1$.
\end{enumerate*}
\end{proposition}

\begin{proof} Let $S$ be any ring with UGN. By \cite[Theorem 3.2$'$]{Chou}, $G$ contains a subsemigroup $X$ that is free on a set of cardinality two. Set $R:=T_G(X,S)$. Since $X$ is not amenable with respect to $G$, Proposition 3.4 implies ${\rm gn}(R)=1$. For every $g\in G$, let $R_g\subseteq R$ be the set of all $X\times X$ matrices $A$ with entries in $S$ such that $A(x,y)=0$ if $x\neq gy$. This provides us with a $G$-grading on $R$. Observe further that $R_1$ is the collection of all diagonal matrices in $T_G(X,S)$. Hence, according to Lemma 2.11, $R_1$ has UGN. 

To establish property (i), let $g\in G$ and let $M$ be the $R_1$-module consisting of all $gX\times X$ matrices $A$ with entries in $S$ such that $A(x,y)=0$ if $x\neq gy$. Notice that $M$ is a free $R_1$-module whose generator is the matrix $A_0\in M$ such that $A_0(gy,y)=1$ for all $y\in X$.  Next define $P$ to be the submodule of $M$ consisting of all the matrices with $0$ in the $(gy,y)$ position whenever $gy\notin X$. Furthermore, define $Q$ to be the submodule of $M$ consisting of all the matrices with $0$ in the $(gy,y)$ position whenever $gy\in X$. Then $M\cong P\oplus Q$ and $P\cong R_g$, which proves (i).
\end{proof}

Now we apply the same trick as in the proof of Theorem~4.2 to prove Theorem~4.6. 

\begin{proof} [Proof of Theorem 4.6] According to Proposition 4.7, there is a ring $S$ graded by $G$ with the following properties.
\begin{itemize}
\item For each $g\in G$, the $S_1$-module $S_g$ is a direct summand in $S_1$. 
\item $S_1$ has UGN.
\item ${\rm gn}(S)=1$.
\end{itemize}

Appealing to Proposition 2.12, we let $T$ be a nonzero ring with ${\rm gn}(T)=n$.  Put $R:=S\times T$ and define a $G$-grading on $R$ by letting $R_g:=S_g\times T$ for every $g\in G$. This grading is plainly full. Also, by Lemma 2.11, $R_1=S_1\times T$ has UGN and ${\rm gn}(R)=n$.  To show (ii), let $g\in G$ and let $A$ be an $S_1$-module such that $S_g\oplus A\cong S_1$. Now make $A$ into an $R_1$-module by defining $a\, (s,t):=as$ for all $a\in A$, $s\in S_1$, and $t\in T$. It is then straightforward to verify that $R_g\oplus A\cong R_1$ as $R_1$-modules, thus completing the proof.
\end{proof}

\begin{remark}{\rm The rings described in the proofs of Proposition~4.7 and Theorem~4.6 are not strongly graded.}
\end{remark}

\subsection{Gradings by finite groups}

Next we show that Theorem B, Corollary A, and Theorem C fail to extend to finite generating numbers. 

\begin{proposition} Let $m, n$ be positive integers with $m\leq n$, and let $G$ be a group that has order $n$.
Then there are a ring $R$ and a skew group ring $R\ast G$ such that ${\rm gn}(R)=n$ and ${\rm gn}(R\ast G)=m$. 
\end{proposition} 

\begin{proof} By Proposition 2.12, there exist rings $S$ and $T$ with generating numbers $m$ and $n$, respectively.  Put $U:=M_n(T)$. According to Lemma 2.9, we have ${\rm gn}(U)=1$. Let $G$ be a finite group of order $n$ and notice that $T(G,T)\cong U$ as rings. Proposition~2.14 implies, then, that $U$ is isomorphic to a skew group ring of $G$ over $\prod_G T$. Write $V:=U\times S$  and $R:=\left (\prod_G T\right )\times S$. Define a $G$-grading on $V$ by setting $V_g:=U_g\times S$ for all $g\in G$. Then this grading makes $V$ into a skew group ring of $G$ over $R$.  Furthermore, applying Lemma 2.11, we obtain ${\rm gn}(R)=n$ and ${\rm gn}(V)=m$.
\end{proof}

Below we use Abrams's reasoning in \cite[Theorem~B]{Abrams} to construct another family of rings that demonstrates that Theorem C does not extend to finite generating numbers. In contrast to the one described in Proposition~4.8, the grading groups for this family comprise all nontrivial finite groups.  

\begin{proposition} Let $G$ be a nontrivial finite group and $m, n$ positive integers with $m\leq n$. Then there exists a ring $R$ strongly graded by $G$ such that
${\rm gn}(R_1)=n$ and ${\rm gn}(R)=m$. 
\end{proposition}

\begin{proof}Invoking Proposition 2.12, we let $S$ be a ring with ${\rm gn}(S)=n$. Then ${\rm gn}(M_n(S))~=~1$ by Lemma~2.9. Put $k:=|G|$ and let $l\in \mathbb Z^+$ such that $nl>k-1$. Set $p:=nl-k+1$, and define a family $\{A_g\, :\, g\in G\}$ of $S$-modules as follows: $A_1:=S^p$; $A_g:=S$ for $g\neq 1$. Write $A:=\bigoplus_{g\in G}A_g$ and $T:={\rm End}_S(A)$. This means $T\cong M_{nl}(S)$, implying that there is a ring embedding $M_n(S)\to T$. Therefore, by Lemma~2.7, we have ${\rm gn}(T)=1$. 

The ring $T$ may be viewed as the ring of all $G\times G$ matrices $P$ such that $P(g,h)\in {\rm Hom}_S(A_h,A_g)$ for every pair $g, h\in G$. This allows us to equip $T$ with a $G$-grading by, for each $g\in G$, defining $T_g$ to be the set of all matrices $P\in T$ such that
$P(x,y)=0$ if $y\neq g^{-1}x$. Moreover, it is straightforward to check that this grading is strong  (see \cite[Proposition~1.1.1(3)]{NO}). 
 
We observe that the following isomorphic relations between rings hold:
\[T_1\cong {\rm End}_S(S^p)\times \underbrace{{\rm End}_S(S)\times \cdots \times {\rm End}_S(S)}_{k-1}\cong M_p(S)\times \underbrace{S\times \cdots \times S}_{k-1}.\]
 Since there is a ring embedding $S\to M_p(S)$, it follows from Lemma 2.7 that ${\rm gn}(M_p(S))\leq~n$. Hence we obtain ${\rm gn}(T_1)=n$ by applying Lemma~2.11. 
Now let $U$ be a ring with ${\rm gn}(U)=m$. Put $R:=T\times U$ and endow $R$ with a strong $G$-grading by defining $R_g:=T_g\times U$ for $g\in G$. Then Lemma~2.11 implies 
${\rm gn}(R_1)=n$ and ${\rm gn}(R)=m$. 
\end{proof}

\begin{remark}{\rm The ring $R$ described in the proof of Proposition 4.9 will not be a crossed product of $G$ over $R_1$ if $p>1$; indeed, $R$ will not even be freely graded. Moreover, it is not hard to see that it will not always be possible to find a ring $R$ satisfying the conclusion of Proposition~4.9 that is a crossed product of $G$ over $R_1$. 
To illustrate this, choose $m, n, k$ to be positive integers such that $k$ divides $n$ and $m<k$. Also, take $G$ to be a group of order $n/k$ and $R_1$ to be a ring with generating number $n$. Now let $R$ be a crossed product of $G$ over $R_1$. Then $R$ can be embedded in ${\rm End}_{R_1}(R)\cong M_{n/k}(R_1)$. Therefore, by Lemmas~2.9 and 2.7, we have ${\rm gn}(R)\geq k>m$. 
}
\end{remark}

We conclude the paper by applying the methods employed for \cite[Theorem~ A]{Abrams} and \cite[Theorem~3.1]{Loc} to prove Theorem~4.10, thus demonstrating that the hypothesis that the ring is projectively graded in Theorem~C cannot be omitted. 

\begin{theorem} Let $G$ be a nontrivial finite group. Then, for any integer $n>0$, there exists a $G$-graded ring $R$ such that
$R_1$ has UGN and ${\rm gn}(R)=n$.
\end{theorem}

Like the approach adopted in both \cite{Abrams} and \cite{Loc}, ours is based upon Theorem~2.13. In addition, we rely on Lemmas~4.11 and 4.12 below. The first, inspired by \cite[Proposition~3.3]{NTO}, describes an elementary and well-known property of projective modules that can be used to determine the generating numbers of their endomorphism rings (see \cite[Theorem 2.35]{Facchini} for a far-reaching generalization of Lemma 4.11).  The second lemma introduces a certain abelian monoid to which we will apply Theorem 2.13 when we prove Theorem~4.10. 

\begin{lemma} Let $R$ be a ring and $P$ a projective $R$-module.  Write $S:={\rm End}_R(P)$. Then the following two statements are equivalent for any $m, n\in \mathbb Z^+$.
\begin{enumerate*}
\item There is an $S$-module epimorphism $S^n\to S^m$. 
\item There is an $R$-module epimorphism $P^n\to P^m$. 
\end{enumerate*}
\end{lemma}

\begin{proof} Assume that there is an $S$-module epimorphism $S^n\to S^m$. This 
induces an $R$-module epimorphism $S^n\otimes_S P\to S^m\otimes_S P$. However, $S\otimes_S P\cong P$ as $R$-modules, so that we have an $R$-module epimorphism $P^n\to P^m$. This proves (i)$\implies$(ii).

Assume that there is an $R$-module epimorphism $\phi:P^n\to P^m$. Since $P$ is projective, the map $\phi$ induces an $S$-module epimorphism ${\rm Hom}_R(P,P^n)\to {\rm Hom}_R(P,P^m)$. But, for every $k\in \mathbb Z^+$, ${\rm Hom}_R(P,P^k)\cong S^k$ as $S$-modules. Therefore we have an $S$-module epimorphism $S^n\to S^m$. This establishes (ii)$\implies$(i).
\end{proof}

\begin{lemma} For any triple $n, k, l$ of positive integers, let $M(n,k,l)$ be the abelian monoid generated by $u, x_1,\dots, x_l, y_1,\dots, y_l$ subject only to the relations
\[ (n + k)(u+x_1 + \cdots + x_l) = n(u+x_1+\cdots + x_l), \ \ \ \ x_i+y_i=u\ \mbox{\rm for}\ i=1,\dots, l.\]
Then, for any $j=1,\dots, l$ and any two positive integers $\lambda, \mu$, 
 \[\lambda x_j\leq \mu x_j\implies \lambda\leq \mu.\] 
\end{lemma}

\begin{proof}  We will make use of the monoid homomorphism $\phi:M(n,k,l)\to C(n,k)$ such that $\phi(u):=\phi(y_i):=a$ and $\phi(x_i):=0$ for $i=1,\dots , l$. In addition, we will employ the monoid homomorphisms $\psi_j:M(n,k,l)\to \mathbb Z$, for $j=1,\dots,l$, defined by 
\[ \psi_j(u):=-1, \ \ \psi_j(x_i):=\begin{cases} 1 &\ \mbox{if}\ i=j\\ 0 &\ \mbox{if}\ i\neq j\end{cases}, \ \ \mbox{and}\ \psi_j(y_i):=\begin{cases} -2 &\ \mbox{if}\ i=j\\ -1 &\ \mbox{if}\ i\neq j\end{cases}.\]

The hypothesis $\lambda x_j\leq \mu x_j$ implies that there are nonnegative integers $\lambda_1,\dots, \lambda_l, \alpha, \beta_1,\dots,\beta_l$ such that $\lambda_j\geq \lambda$ and
\begin{equation} \lambda_1x_1 + \cdots + \lambda_lx_l+ \alpha u + \beta_1 y_1 + \cdots + \beta_ly_l = \mu x_j.\end{equation}
Applying $\phi$ to equation (4.1) yields $(\alpha + \beta_1 + \cdots + \beta_l)a=0$, which means $\alpha=\beta_i=0$ for $i=1,\dots,l$. 
We now apply $\psi_j$ to (4.1), obtaining $\lambda_j=\mu$. Thus $\lambda\leq \mu$. 
\end{proof}

\begin{proof}[Proof of Theorem 4.10] We are given an arbitrary nontrivial finite group $G$ and an arbitrary positive integer $n$. Set $l:=|G|-1$. Taking $k$ to be another arbitrary positive integer, let $M(n,k,l)$ be the abelian monoid defined in Lemma 4.12. The monoid homomorphisms $\phi:M(n,k,l)\to C(n,k)$ and $\psi_j:M(n,k,l)\to \mathbb Z$, for $j=1,\dots, l$, used in the proof of that lemma can also be employed to show that, if $\alpha_1,\dots\alpha_l, \beta, \gamma_1,\dots,\gamma_l$ are nonnegative integers, then
\[\sum_{i=1}^l \left (\alpha_i x_i+\beta u+\gamma_i y_i\right )=0\ \ \ \Longleftrightarrow\ \ \ \alpha_i=\beta=\gamma_i=0\ \  \ \ \ \forall\, i=1,\dots, l.\]
This observation, combined with the fact that, for every $t\in M(n,k,l)$, $t\leq \lambda u$ for some $\lambda \in \mathbb Z^+$, allows us to invoke Theorem 2.13. 
 We thereby acquire a ring $S$  such that there is a monoid isomorphism $\theta:M(n,k,l)\to {\rm Proj}(S)$ with $\theta(u)=[S]$. 
 
For each $i=1,\dots, l$, let $X_i$ be a finitely generated projective $S$-module such that $[X_i]=\theta(x_i)$. Also, let $\sigma:G-\{1\}\to \{1,\dots,l\}$ be a bijection. For each $g\in G$, write $A_g:=X_{\sigma(g)}$ for $g\neq 1$ and $A_1:=S$. Put $A:=\bigoplus_{g\in G}A_g$ and $R:={\rm End}_S(A)={\rm End}_S(S\oplus \bigoplus_{i=1}^l X_i)$. 
The ring $R$ may be viewed as the ring of all $G\times G$ matrices $P$ such that $P(g,h)\in {\rm Hom}_S(A_h,A_g)$ for every pair $g, h\in G$. We can therefore endow $R$ with a $G$-grading by, for each $g\in G$, defining $R_g$ to be the set of all matrices $P\in R$ such that
$P(x,y)=0$ if $y\neq g^{-1}x$.  Notice that we then have
\[R_1\cong S\times {\rm End}_S(X_1)\times\cdots \times {\rm End}_S(X_l)\]
 as rings. 

 Let $i=1,\dots, l$. Write $T_i:={\rm End}_S(X_i)$ and let $p\in \mathbb Z^+$. Lemma 4.12 implies $(p+1)x_i\nleq px_i$. As a result, $X_i^{p+1}$ is not an $S$-module direct summand in $X_i^p$. Hence, by Lemma~4.11, there is no $T_i$-module epimorphism $T_i^p\to T_i^{p+1}$. In other words, $T_i$ has UGN for $i=1,\dots, l$. Thus, by Lemma~2.11 (or Lemma~2.7),  $R_1$ possesses UGN.  
 
 It remains to show that ${\rm gn}(R)=n$. To accomplish this, we observe that $n$ is the smallest positive integer such that, in the cyclic monoid $C(n,k)$, $(n+1)a\leq na$. Hence, referring to the homomorphism $\phi:M(n,k,l)\to C(n,k)$, we conclude that $n$ is also the smallest positive integer such that $(n+1)(u+\sum_{i=1}^l x_i)\leq n(u+\sum_{i=1}^l x_i)$ in $M(n,k,l)$. Therefore the integer $n$ may be further characterized as the smallest positive integer such that $\left (S\oplus \bigoplus_{i=1}^l X_i\right )^{n+1}$ is an $S$-module direct summand in  $\left (S\oplus \bigoplus_{i=1}^l X_i\right )^{n}$. An appeal to Lemma~4.11, then, permits us to conclude that $n$ is also the smallest positive integer such that there is an $R$-module epimorphism $R^n\to R^{n+1}$. In other words, we have ${\rm gn}(R)=n$. 
\end{proof}

\vspace{10pt}

\noindent {\sc Karl Lorensen}\\
Department of Mathematics and Statistics\\
Pennsylvania State University, Altoona College\\
Altoona, PA 16601, USA\\
E-mail: {\tt kql3@psu.edu}
\vspace{20pt}

\noindent {\sc Johan \"Oinert}\\
Department of Mathematics and Natural Sciences\\
Blekinge Institute of Technology\\
SE-37179 Karlskrona, Sweden\\
E-mail: {\tt johan.oinert@bth.se}

\end{document}